\documentclass[11pt,reqno]{amsart}
\usepackage{graphicx, amssymb, color,pdfsync}
\usepackage[all,cmtip]{xy}
\numberwithin{equation}{section}
\usepackage{mathrsfs}
\usepackage{hyperref}
\usepackage{multicol}

\hypersetup{
    colorlinks=true,       
    linkcolor=blue,          
    citecolor=blue,        
    filecolor=blue,      
    urlcolor=blue           
}
\usepackage{tikz}
\usetikzlibrary{matrix,arrows}
\DeclareMathOperator{\Aut}{Aut}
\usepackage[switch]{lineno}
\usepackage{yfonts}
\advance\hoffset by -0.9 truecm

\DeclareMathOperator{\Gal}{{Gal}}

\DeclareMathOperator{\Sp}{Sp}

\DeclareMathOperator{\sy}{\tiny \mbox{sym}}

\begin{document}
\newcommand{\s}{\vspace{0.2cm}}

\newtheorem{theo}{Theorem}
\newtheorem{propn}{Proposition}[section]
\newtheorem{prop}{Proposition}
\newtheorem{coro}{Corollary}
\newtheorem{lemm}{Lemma}
\newtheorem{claim}{Claim}
\newtheorem{example}{Example}
\theoremstyle{remark}
\newtheorem*{rema}{\it Remarks}
\newtheorem*{rema1}{\it Remark}
\newtheorem{remark}{\bf Remark}
\newtheorem*{defi}{\it Definition}
\newtheorem*{theo*}{\bf Theorem}
\newtheorem*{coro*}{Corollary}

\title[On large prime actions on Riemann surfaces]{On large prime actions on Riemann surfaces}
\date{}
\author{Sebasti\'an Reyes-Carocca}
\address{Departamento de Matem\'atica y Estad\'istica, Universidad de La Frontera, Avenida Francisco Salazar 01145, Temuco, Chile.}
\email{sebastian.reyes@ufrontera.cl}
\author{Anita M. Rojas}
\address{Departamento de Matem\'aticas, Facultad de Ciencias, Universidad de Chile. Las Palmeras 3425, \~Nu\~noa, Santiago, Chile.}
\email{anirojas@uchile.cl}

\thanks{Partially supported by Fondecyt Grants 1180073, 11180024, 1190991 and Redes Grant 2017-170071}
\keywords{Riemann surfaces, group actions, automorphisms, Jacobian varieties}
\subjclass[2010]{30F10, 14H37, 30F35, 14H30, 14H40}

\begin{abstract} In this article we study  compact Riemann surfaces of genus $g$ with an automorphism of prime order $g+1.$ The main result provides a classification of such surfaces. In addition, we give a description of them as algebraic curves, determine and realise their full automorphism groups and compute their fields of moduli. We also study some aspects of their Jacobian varieties such as   isogeny decompositions and complex multiplication. Finally, we determine the period matrix of the Accola-Maclachlan curve of genus  four.
\end{abstract}
\maketitle
\thispagestyle{empty}
\section{Introduction and statement of the results}

Let $\mathscr{M}_g$ denote the moduli space of compact Riemann surfaces (smooth irreducible complex  algebraic curves) of genus $g \geqslant 2.$ It is classically known that $\mathscr{M}_g$ is endowed with a structure of complex analytic space of dimension $3g-3,$ and that for $g \geqslant 4$ its  singular locus agrees with the branch locus of the canonical  projection
$$T_g \to \mathscr{M}_g$$
where $T_g$ stands for the Teichm\"{u}ller space of genus $g.$ In other words, if $g \geqslant 4$ then
$$\mbox{Sing}(\mathscr{M}_g)=\{[S] \in \mathscr{M}_g : \mbox{Aut}(S)  \neq 1 \}$$where $\mbox{Aut}(S)$ denotes the full automorphism group of $S.$

The classification of groups of automorphisms of compact Riemann surfaces is a classical  problem which has  attracted broad interest  ever since it was proved that the full automorphism group of a compact  Riemann surface $S$ of genus $g \geqslant 2$ is finite, and that $$| \mbox{Aut}(S) | \leqslant 84(g-1).$$

It is well-known that there are infinitely many values of $g$ for which there is no compact Riemann surfaces of genus $g$ possessing $84(g-1)$ automorphisms. Regarding this matter,  Accola \cite{Accola} and Maclachlan \cite{Mac}  proved that, for fixed $g,$ the largest order $n_0(g)$ of the full automorphism group of a compact Riemann surface of genus $g$ satisfies \begin{equation}\label{ine}n_0(g)  \geqslant 8(g+1)\end{equation}and that for infinitely many values of $g$ the inequality  \eqref{ine} turns into an equality. 

\s

We denote by $X_8$ the  so-called Accola-Maclachlan curve, namely the compact Riemann surface of genus $g$ with $8(g+1)$ automorphisms given by  the algebraic  curve   \begin{equation*} \label{eAccola}y^2=x^{2(g+1)}-1.\end{equation*}

The Accola-Maclachlan curve is a remarkable example of a compact Riemann surface  determined by the order of its full automorphism group. More precisely, Kulkarni in \cite{K1} succeeded in proving that, up to finitely many values of the genus, if $g \not\equiv 3 \mbox{ mod } 4$ then $X_8$ is the unique compact Riemann surface of genus $g$ with exactly $8(g+1)$ automorphisms. 

\s

The analogous problem of finding $n_0(g)$ but for uniparametric families of compact Riemann surfaces was studied in \cite{CI}. Concretely, it was  proved the existence of a closed equisymmetric complex one-dimensional family, henceforth denoted by $\bar{\mathscr{C}}_g,$ of hyperelliptic compact Riemann surfaces  of genus $g$  with a group of automorphisms isomorphic to $$\mathbf{D}_{g+1} \times C_2 \, \mbox{ acting with signature } (0; 2,2,2,g+1)$$(we shall recall the precise definition of {\it signature} in \S \ref{lpi} and \S\ref{lpi2}). It was then shown that $4(g+1)$ is the largest order of the full automorphism group of complex one-dimensional families of  compact Riemann surfaces of genus $g$ appearing for all $g.$
These results were recently extended to the three and four-dimensional case in \cite{families} while the two-dimensional case is derived from the results of \cite{yoisrael}.

\s

It is a well-known fact that if  a compact Riemann surface of genus $g \geqslant 2$ 
 has an automorphism of  prime order $q$ such that $q > g$ then either $q=2g+1$ or $q=g+1.$ The former case corresponds to the so-called Lefschetz surfaces. This paper deals with the latter case.

\s

Let $q \geqslant 5$ be a prime number. Consider the singular sublocus $$\mathscr{M}_{q-1}^{q} \subset \mbox{Sing}(\mathscr{M}_{q-1})$$consisting of all those compact Riemann surfaces of genus $q-1$ endowed with an automorphism of order $q.$ This sublocus was studied  by Urz\'ua in \cite{urzua}  from a hyperbolic geometry point of view, and later by Costa and Izquierdo in \cite{scand} when the existence of complex one-dimensional isolated strata of the singular locus of the moduli space was proved.  

\s

This paper is devoted to  classify and describe the surfaces lying in $\mathscr{M}_{q-1}^{q}$ and to study some aspects of the corresponding Jacobians in the singular locus of the moduli space of principally polarised abelian varieties of dimension $q-1$. In other words,  we shall consider all those compact Riemann surfaces of genus $g \geqslant 4$ (and their Jacobian varieties)  with a group of automorphisms of order $$\lambda(g+1) \mbox{ where } \lambda \geqslant 1 \mbox{ is an integer,}$$ under the assumption that $q:=g+1$ is a prime number. 

\subsection*{\it The classification} The first result of the paper provides a classification of these surfaces.

\begin{theo} \label{tm}
Let $q \geqslant 7$ be  a  prime number. If $S$ is a compact Riemann surface of genus $g=q-1$ endowed with a group of automorphisms of order $\lambda q$ for some integer $\lambda \geqslant 1,$ then $$\lambda \in \{1,2,3,4,8\}.$$

Assume $\lambda=8.$ Then $S$ is isomorphic to the Accola-Maclachlan curve $X_8.$

\s
Assume $\lambda=4.$
\begin{enumerate}
\item If $q \equiv 3 \mbox{ mod } 4$ then $S$ belongs to the closed family $\bar{\mathscr{C}}_g.$ 
\s

\item If $q \equiv 1 \mbox{ mod } 4$ then $S$  belongs to the closed family $\bar{\mathscr{C}}_g$ or $S$ is isomorphic to the unique   compact Riemann surface  $X_4$ with full automorphism group isomorphic to $$C_{q} \rtimes_4 C_4 \mbox{ acting with signature } (0; 4,4,q).$$
\end{enumerate}Moreover, if $\mathscr{C}_g$ stands for the interior of $\bar{\mathscr{C}}_g$ then $$\bar{\mathscr{C}}_g - \mathscr{C}_g=\{X_8\}.$$

Assume $\lambda=3.$ Then $S$ is isomorphic to the unique compact Riemann surface $X_3$ with full automorphism group isomorphic to $$C_{q} \times C_3 \mbox{ acting with signature } (0; 3,q,3q).$$

Assume $\lambda=2.$ Then one of the following statements holds.

\begin{enumerate}
\item $S$ is isomorphic to one of the $\tfrac{q-3}{2}$ pairwise non-isomorphic  compact Riemann surfaces $X_{2,k}$ for $k \in \{1, \ldots, \tfrac{q-3}{2}\}$ with full automorphism group isomorphic to $$C_{q} \times C_2  \mbox{ acting with signature } (0; q,2q,2q).$$

\item $S$ belongs to the closed family $\bar{\mathscr{K}}_g$ of compact Riemann surfaces with a group of automorphisms isomorphic to $$\mathbf{D}_q  \mbox{ acting with signature }  (0; 2,2,q,q).$$
\end{enumerate}Moreover, the closed family $\bar{\mathscr{K}}_g$ consists of at most \begin{displaymath}
 \left\{ \begin{array}{ll}
\tfrac{q+3}{4} & \textrm{if $q \equiv 1 \mbox{ mod } 4$}\\
\tfrac{q+1}{4}  & \textrm{if $q \equiv 3 \mbox{ mod } 4$}
  \end{array} \right.
\end{displaymath}  equisymmetric strata;  one of them being ${\mathscr{C}}_g.$ Furthermore, if $\mathscr{K}_g$ stands for the interior of $ \bar{\mathscr{K}}_g$ then the full automorphism group of  $S \in \mathscr{K}_g-\mathscr{C}_g$ is isomorphic to $\mathbf{D}_q$ and\begin{displaymath}
 \bar{\mathscr{K}}_g-\mathscr{K}_g= \left\{ \begin{array}{ll}
\{X_4,X_8\} & \textrm{if $q \equiv 1 \mbox{ mod } 4$}\\
 \,\,\,\,\{X_8\}  & \textrm{if $q \equiv 3 \mbox{ mod } 4.$}
\end{array} \right.
\end{displaymath}
\end{theo}

\s
\begin{remark} We point out some observations concerning Theorem \ref{tm}.
\begin{enumerate}
\s
\item If $S \in \mathscr{M}_{q-1}^q$ then either $\mbox{Aut}(S) \cong C_q$ or $S$ lies in one of the cases described in the  theorem. These two possible situations were considered in \cite{scand} where the focus was put on finding isolated equisymmetric strata of $\mbox{Sing}(\mathscr{M}_{g})$. We shall  discuss later the results of \cite{scand} in terms of our terminology (see Remark \ref{casoq} in \S\ref{ptm}). 
\s
\item The case $q=5$ is slightly different. As a matter of fact, if $S$ has genus $4$  and is endowed with a group of automorphisms of order $5\lambda$ for some $\lambda \geqslant 1$ then, in addition to the case $\mbox{Aut}(S)\cong C_5$ and the possibilities given in the  theorem,  $\lambda$ can  equal  $12$ and $24.$ In the last two cases, $S$ is isomorphic to the classical Bring's curve; see \cite{CIz} and \cite{KK}.

\s

\item We conjecture that the upper bound given in the theorem  for the number of equisymmetric strata of the closed  family $\bar{\mathscr{K}}_g$ is sharp. By means of computer routines developed  in \cite{poly2}, it can be seen the sharpness of the bound for small primes ($q \leqslant 23$). 

\s

\item We emphasise that the equisymmetric family $\bar{\mathscr{C}}_g$ is contained in the family $\bar{\mathscr{K}}_g$. 

\s

\item An analogous  classification as in the theorem but for the compact Riemann surfaces lying in $\mathscr{M}_{q+1}^q$ was obtained in a  series of articles due to Belolipetsky, Izquierdo, Jones and the first author; see  \cite{BJ}, \cite{IJRC}, \cite{IRC} and \cite{yoisrael}. 
\end{enumerate}
\end{remark}

\subsection*{\it Algebraic description} Although the literature still shows few general results in this direction, there is a great interest in providing   descriptions of compact Riemann surfaces as  algebraic curves in an explicit manner. The following result gives such a description for the  surfaces appearing in Theorem \ref{tm}, as well as a realisation of their full automorphism groups.

\begin{theo} \label{t2} Let $q \geqslant 5$ be a prime number and let $g=q-1.$ Set $\omega_l=\mbox{exp}(\tfrac{2\pi i}{l}).$

\s

If $S$ belongs to the closed family $\bar{\mathscr{C}}_g$ then $S$ is isomorphic to the normalisation of the singular affine algebraic curve $$\mathscr{X}_t : y^2=(x^{q}-1)(x^{q}-t) \,\, \mbox{ for some }t \in \mathbb{C}-\{0,1\}.$$
In addition, if $S \in \mathscr{C}_g$  then the full automorphism group of $S \cong \mathscr{X}_t$ is generated by  $$(x,y) \mapsto (\omega_{q}x, -y) \,\, \mbox{ and } \,\, (x,y) \mapsto (\sqrt[q]{t}\tfrac{1}{x}, \sqrt{t}\tfrac{y}{x^{q}}).$$

Assume $q \equiv 1 \mbox{ mod } 4$ and choose $\rho \in \{2, \ldots, q-2\}$ such that $\rho^4\equiv 1 \mbox{ mod } q.$  Then $X_4$ is isomorphic to the normalisation of the singular affine algebraic curve $$y^{q}=(x-1)(x-i)^{\rho}(x+1)^{q-1}(x+i)^{q-\rho}$$where $i^2=-1.$  In the previous model the full automorphism group of $X_4$ is generated by  $$(x,y) \mapsto (x, \omega_q y) \,\, \mbox{ and }\,\, (x,y) \mapsto (ix,\varphi(x)y^{\rho})$$where $\varphi(x)= \tfrac{-(x+i)^{e-\rho}}{(x-i)^{e-1}(x+1)^{\rho -1}}$ and $e = \tfrac{\rho^2+1}{q}.$

\s

$X_3$ isomorphic to the normalisation of the singular affine algebraic curve $$y^3=x^{q}-1$$ and, in this model, its full automorphism group is generated by $(x,y) \mapsto (\omega_{q}x, \omega_3y).$ 
 
\s

For each $k \in \{1, \ldots, \tfrac{q-3}{2}\}$ there exists $n_k \in \{1, \ldots, q-1\}$ different from $q-2$ such that $X_{2,k}$ is isomorphic to the normalisation of the singular affine algebraic curve $$y^q=x^{n_k}(x^2-1)$$ and, in this model, its full automorphism group is generated by $$(x,y) \mapsto (x, \omega_{q}y) \,\, \mbox{ and }\,\, (x,y) \mapsto (-x, (-1)^{n_k}y).$$

\s

If $S$ belongs to the closed family $\bar{\mathscr{K}}_g$,  then $S$ is isomorphic to the normalisation of the singular affine algebraic curve $$\mathscr{Z}_t : y^q=(x-1)(x+1)^{q-1}(x-t)(x+t)^{q-1}\,\, \mbox{ for some }t \in \mathbb{C}-\{0, \pm1\}$$and, if $S \neq X_4$ and $S \notin \bar{\mathscr{C}}_g$  then the full automorphism group of $S \cong \mathscr{Z}_t$ is generated by $$(x,y) \mapsto (x, \omega_q y) \,\, \mbox{ and }\,\, (x,y) \mapsto (-x, \phi_t(x)y^{-1})$$where $\phi_t(x)=(x^2-1)(x^2-t^2).$ 
\end{theo}

The theorem above overlaps results obtained in  \cite[\S 11]{urzua}.

\subsection*{\it Hyperelliptic surfaces} Arakelian and Speziali in \cite{AS} studied groups of automorphisms of large prime order of (non-necessarily smooth)  
projective absolutely irreducible algebraic curves over algebraically closed fields of any characteristic. In terms of our terminology, in \cite[Theorem 4.7]{AS} they proved that if $q \geqslant 7$ is a prime number and $S$ is a compact Riemann surface of genus $q-1$ with a group of automorphisms of order $\lambda q$ then $$S \mbox{ is non-hyperelliptic}  \mbox{ implies }  \lambda \in \{1,2,3,4\}.$$

The following result lengthens  the implication above; it follows from Theorem \ref{tm}.
\begin{prop} Let $q \geqslant 7$ be a prime number.  The compact Riemann surfaces lying in $\mathscr{M}_{q-1}^q$ that are non-hyperelliptic are $X_{2,k}, X_3, X_4,$ the surfaces  which belong to  $\mathscr{K}_g-\mathscr{C}_g$ and the ones for which $\mbox{Aut}(S) \cong C_q.$
\end{prop}

\subsection*{\it Jacobian variety} Let $S$ be a compact Riemann surface of genus $g \geqslant 2.$ We denote by $JS$ the Jacobian variety of $S$, that is, the quotient$$JS=\mathscr{H}^1(S, \mathbb{C})^*/H_1(S, \mathbb{Z}),$$where $\mathscr{H}^1(S, \mathbb{C})^*$ stands for the dual of the $g$-dimensional complex vector space of holomorphic forms of $S$ and $H_1(S, \mathbb{Z})$  stands for the first integral homology group of $S$. 

We emphasise the following two classical facts (see, for example, \cite{BL}):
\begin{enumerate}
\item $JS$ is an irreducible  principally polarised abelian variety of dimension $g,$ and
\item up to isomorphism, the surface is  determined by its Jacobian (Torelli's theorem).
\end{enumerate}

If $G$ is a group acting on  $S$ then $G$ also acts on $JS$ and this action, in turn, induces the so-called group algebra decomposition of $JS$. Concretely$$JS\sim  A_1 \times \cdots \times A_r \sim  B_1^{n_1} \times \cdots \times B_r^{n_r}$$
where the factors $A_j$ are pairwise non-$G$-isogenous abelian subvarieties of $JS$ uniquely determined, and in correspondence with central idempotents generating the simple algebras decomposing the rational group algebra  of $G.$  Each $A_j$ decomposes further as $B_j^{n_j}$ where the abelian subvarieties  $B_j$ are no longer unique and  are related to the decomposition of each simple algebra as a product of minimal left ideals.  The numbers $r$ and $n_j$ depend only on the algebraic structure of $G.$ See \cite{cr} and \cite{l-r}.
\s

The following result provides the group algebra  decomposition of the  Jacobian varieties of the 
 surfaces of Theorem \ref{tm}, with the exception of $X_3$ and $X_{2,k}$. In fact, the group algebra decomposition of $JX_3$ is trivial whilst the one of $JX_{2,k}$ agrees with the classical decomposition $$JX_{2,k} \sim J(X_{2,k}/H) \times \mbox{Prym}(X_{2,k} \to X_{2,k}/H)$$where $\mbox{Prym}$ stands for the Prym variety and $H \leqslant \mbox{Aut}(X_{2,k})$ is isomorphic to $C_2$.

\begin{theo} \label{t3}
Let $q \geqslant 5$ be a prime number and let $g=q-1.$ 

\s

The Jacobian variety $JX_8$  decomposes, up to isogeny, as the square power $$JX_8 \sim JY_8^2$$where $Y_8$ is quotient  compact Riemann surface given by the action of $\langle   z \rangle$ on $X_8,$ where $$\operatorname{Aut}(X_8) \cong \langle x,y,z: x^{2q}=y^2=z^2=1, [x,y]=[z,y]=1, zxz=x^{-1}y \rangle.$$ 

\s

The Jacobian variety $JX_4$ of $X_4$ decomposes, up to isogeny, as the fourth power $$JX_4 \sim JY_4^4$$where $Y_4$ is quotient  compact Riemann surface given by the action of $\langle B \rangle$ on $X_4,$ where  $$\operatorname{Aut}(X_4) \cong \langle A, B: A^{q}=B^4=1, BAB^{-1}=A^{\rho} \rangle$$and $\rho$ is a primitive fourth root of unity in $\mathbb{Z}_q.$ 

\s

The Jacobian variety $JS$ of $S \in \mathscr{K}_g$ decomposes, up to isogeny, as the square power $$JS \sim JX^2$$where $X$ is quotient  compact Riemann surface given by the action of $\langle  s \rangle$ on $S,$ where

\begin{displaymath}
\operatorname{Aut}(S) \cong  \left\{ \begin{array}{ll}
\hspace{0.5 cm} \mathbf{D}_{q} & \textrm{if $S \in \mathscr{K}_g-\mathscr{C}_g$}\\
 \mathbf{D}_{q} \times C_2 & \textrm{if $S \in \mathscr{C}_g$}
  \end{array} \right.
\end{displaymath}and $\mathbf{D}_{q}=\langle r,s : r^{q}=s^2=(sr)^2=1\rangle.$
\end{theo}

\subsection*{\it Field of moduli and fields of definition} Let $\mbox{Gal}(\mathbb{C}/\mathbb{Q})$ denote the group of field automorphisms of $\mathbb{C}.$ The correspondence $$\mbox{Gal}(\mathbb{C}/\mathbb{Q}) \times \mathscr{M}_g \to \mathscr{M}_g \,\, \mbox{ given by }\,\, (\sigma, [S]) \mapsto [S^{\sigma}]$$where $S^{\sigma}$ is the Galois $\sigma$-transformed of $S$ (considered as algebraic curve) defines an action. 

 The field of moduli of a compact Riemann surface $S$ is the fixed field $\mathcal{M}(S)$ of the isotropy group of $S$ under the aforementioned action, namely $$\mathcal{M}(S)=\mbox{fix}\{\sigma \in \mbox{Gal}(\mathbb{C}/\mathbb{Q}): S^{\sigma} \cong S\}.$$
 
The field of moduli of $S$ agrees with the intersection of all its fields of definition and, as proved by Koizumi in \cite{koi}, $S$ can be defined over a finite degree extension of $\mathcal{M}(S).$ 
 
\s
 
Necessary and sufficient conditions under which $S$   can be defined over its field of moduli were provided by Weil in \cite{W1} (see also \cite{W2} for a constructive proof of Weil's theorem); these conditions are trivially satisfied if $S$ has no non-trivial automorphisms. Besides, as proved by Wolfart in \cite{JW}, if $S$ is quasiplatonic  then $S$ can be defined over its field of moduli. 

\s

The general question of deciding whether or not the field of moduli is a field of definition is a challenging problem; see, for example, \cite{Badr}, \cite{YoGa}, \cite{Hid}, \cite{FM}, \cite{Huggins}, \cite{Kont} and \cite{yomoduli}. In this direction, it is a known fact that if the genus of $S/\mbox{Aut}(S)$ is zero then either $S$ can be defined over $\mathcal{M}(S)$  or  over a quadratic extension of it; see \cite{DE} and also \cite{Hidu} for recent results.

\s

We now study the aforementioned problem for the compact Riemann surfaces of Theorem \ref{tm}. First, note that for the quasiplatonic ones the problem is trivial. Indeed:

\begin{enumerate}
\s
\item As proved in Theorem \ref{t2}, the surfaces $X_3, X_{2,k}$ and $X_8$ are defined over $\mathbb{Q}$ and therefore their fields of moduli are $\mathbb{Q}.$  
\s

\item As mentioned above, the fact that $X_4$ is quasiplatonic implies that it  can be defined over its field of moduli. Moreover,  the uniqueness of $X_4$ implies that its field of moduli is $\mathbb{Q}.$  In fact, we shall see later (Remark \ref{remax4} in \S\ref{sprooft2}) that $X_4$ is isomorphic to $$y^{q}=x(x+1)^{\rho}(x-1)^{q-\rho}$$
\end{enumerate}

The remaining cases (that is, the surfaces lying in the family $\mathscr{K}_g$ since it contains $\mathscr{C}_g$)  are given in the following proposition. 

\begin{prop} \label{moduli} Let $q \geqslant 5$ be prime  and let $g=q-1.$ If $S$ belongs to the family ${\mathscr{K}}_g$ and  $$S \cong \mathscr{Z}_t =\{ (x,y):  y^q=(x-1)(x+1)^{q-1}(x-t)(x+t)^{q-1}\}$$for $t \in \mathbb{C}-\{0, \pm 1\}$ then  the field of moduli of $S$ is $\mathbb{Q}(t)$.
\end{prop}

It it worth mentioning that a compact Riemann surface and its Jacobian variety can be defined over the same fields and that their fields of moduli agree; see \cite{ski} and also \cite{Milne}. 

The following result is a direct consequence of the above. 

\begin{coro}
The compact Riemann surfaces of Theorem \ref{tm} and their Jacobian varieties can be defined over their fields of moduli.
\end{coro}

\subsection*{\it The sublocus of $\mathcal{A}_g$ with $G$-action} It is well-known that the moduli space $\mathcal{A}_g$ of principally polarised abelian varieties  of dimension $g$ is isomorphic to the quotient $$ \pi : \mathscr{H}_g \to \mathcal{A}_g  \cong \mathscr{H}_g/\Sp(2g, \mathbb{Z})$$of the Siegel upper half-space $\mathscr{H}_g$ by the action of the symplectic group $\Sp(2g, \mathbb{Z})$. If  the isomorphism class of $JS$ is represented by $Z_S \in \mathscr{H}_g$ then there is an isomorphism of groups $$\mbox{Aut}(JS) \cong \Sigma_S:=\{ R \in \Sp(2g, \mathbb{Z}) : R \cdot Z_S=Z_S\},$$where $\Sigma_S$ is well-defined up to conjugation in $\Sp(2g, \mathbb{Z}).$ The subset of $\mathscr{H}_g$ given by $$\mathscr{S}_S:=\{Z\in \mathscr{H}_g: R \cdot Z=Z \text{ for all } R \in \Sigma_S\}$$consists of those matrices representing principally polarised abelian varieties of dimension $g$ admitting an action  which is equivalent to the one of $\mbox{Aut}(JS).$  This subset is, indeed, an analytic submanifold of $\mathscr{H}_g$  closely related with some special subvarieties and Shimura families of $\mathcal{A}_g$.

Observe that if $\bar{\mathscr{U}}_g$ is an equisymmetric  family of compact Riemann surfaces of genus $g$ and if $S$ is any surface lying in the interior  $\mathscr{U}_g$  of $\bar{\mathscr{U}}_g$ then $$ \{JX : X \in \mathscr{U}_g \} \subseteq  \pi (\mathscr{S}_S).$$

In general, those loci of $\mathcal{A}_g$ do not agree. Nonetheless, the uncommon cases in which these dimensions do agree have been useful in  finding Jacobians with complex multiplication. 
 
Although a satisfactory description of the matrices in $\mathscr{S}_S$ seems to be a difficult problem, as we shall see in \S\ref{ff2}, there is a simple representation theoretic way to compute the dimension of  the (component which contains $JS$ of) $\mathscr{S}_S$. We shall denote the aforementioned dimension by $N_S.$

\begin{theo} \label{t4}
Let $q \geqslant 5$ be prime,  let $g=q-1$ and let $S \in \mathscr{K}_g.$ Then $$N_{X_8}=N_{X_3}=N_{X_{2,k}}=0, \,\, N_{X_4}=\tfrac{q-1}{4} \,\, \mbox{ and } \,\, N_{S}=\tfrac{q-1}{2}.$$\end{theo}

According to results due to Streit in \cite{ST} (and later generalised in \cite{paola} for higher dimension), if $N_S$ equals zero then the full automorphism group of $S$ determines the period matrix for $JS$ and  $JS$ admits complex multiplication. We refer to \cite{Mu} and \cite{Shaska} for recent applications of this result for quasiplatonic curves that are  hyperelliptic and superelliptic.

As a direct consequence of the previous theorem we recover the following known result.

\begin{coro}
The Jacobian varieties of  $X_3, X_{2,k}$ and $X_8$ admit complex multiplication. 
\end{coro}

In spite of the fact that the problem of determining the period matrix of a given Jacobian variety is, in general, intractable,  interesting results have been obtained for some famous Riemann surfaces. For instance,  the period matrices of the Macbeath's curve of genus seven and of the Bring's curve were determined in \cite{BT}  and  \cite{RR} respectively. A method to find the period matrices of the  Accola-Maclachlan and Kulkarni surfaces was given in \cite{periodos}. In addition, in \cite[Example 3.7]{periodos} the authors went even further and employed their method to provide the period matrix of the 
 Accola-Maclachlan curve of genus two in an explicit way.  
 
 At the end of the paper we  determine explicitly the period matrix of the Accola-Maclachlan curve of genus  four.

\s

This article is organised as follows. In  \S\ref{prelimi} we  succinctly review the basic preliminaries: Fuchsian groups and group action on Riemann surfaces and abelian varieties.  The proof of Theorem \ref{tm} is given in  \S\ref{ptm} and the proofs of Theorem \ref{t2} and Proposition \ref{moduli} are given in  \S\ref{sprooft2}. In  \S\ref{lemmatas} we prove some basic algebraic lemmata needed to prove, in  \S\ref{pt3t4},  Theorems \ref{t3} and \ref{t4}. Finally, we include an addendum in which the period matrix of the Accola-Maclachlan curve of genus  four is computed.

\section{Preliminaries} \label{prelimi}
\subsection{Fuchsian groups} \label{lpi} A {\it Fuchsian group} is a  discrete group of automorphisms of the  upper half-plane $\mathbb{H}.$ If $\Delta$ is a Fuchsian group and the orbit space $\mathbb{H}/{\Delta}$ given by the action of $\Delta$ on $\mathbb{H}$ is  compact, then the algebraic structure of $\Delta$ is determined by its {\it signature}:\begin{equation} \label{sig} \sigma(\Delta)=(\gamma; k_1, \ldots, k_s),\end{equation}where  $\gamma$ is the genus of  $\mathbb{H}/{\Delta}$ and $k_1, \ldots, k_s$ are the branch indices in the universal canonical projection $\mathbb{H} \to \mathbb{H}/{\Delta}.$ In this case, $\Delta$ has a canonical presentation in terms of {\it canonical generators} $\alpha_1, \ldots, \alpha_{\gamma}$, $\beta_1, \ldots, \beta_{\gamma},$ $ x_1, \ldots , x_s$ and relations
\begin{equation}\label{prese}x_1^{k_1}=\cdots =x_s^{k_s}=\Pi_{i=1}^{\gamma}[\alpha_i, \beta_i] \Pi_{i=1}^s x_i=1,\end{equation}where the brackets stand for the commutator. The Teichm\"{u}ller space of $\Delta$ is a complex analytic manifold homeomorphic to the complex ball of dimension $3\gamma-3+s$.

Let $\Delta'$ be a group of automorphisms of $\mathbb{H}$ such that $\Delta \leqslant \Delta'$ of  finite index. Then $\Delta'$ is also Fuchsian and they are related by the so-called Riemann-Hurwitz formula $$ 2\gamma-2 + \Sigma_{i=1}^s(1-\tfrac{1}{k_i})= [\Delta' : \Delta] \cdot  [2\gamma'-2 + \Sigma_{i=1}^r(1-\tfrac{1}{k_i'})].$$where $\sigma(\Delta')=(\gamma'; k_1', \ldots, k_r').$

\subsection{Group action on Riemann surfaces} \label{lpi2} Let $S$ be a compact Riemann surface of genus $g \geqslant 2.$  A finite group $G$ acts on $S$ if there is a group monomorphism $\epsilon: G\to \Aut(S).$ The orbit space $S/G$ given by the action of $G \cong \epsilon(G)$ on $S$ inherits naturally a  Riemann surface structure such that the canonical projection $S \to S/G$ is holomorphic. 

\s

By the classical uniformisation theorem, there is a unique, up to conjugation, Fuchsian group $\Gamma$ of signature $(g; -)$ such that $S \cong \mathbb{H}/{\Gamma}.$ Moreover, $G$ acts on $S$ if and only if there is a Fuchsian group $\Delta$ containing $\Gamma$ together with a group  epimorphism \begin{equation}\label{episs}\theta: \Delta \to G \, \, \mbox{ such that }  \, \, \mbox{ker}(\theta)=\Gamma.\end{equation}

It is said that $G$ acts on $S$ with signature $\sigma(\Delta)$ and that the action is represented by the {\it surface-kernel epimorphism} \eqref{episs}; henceforth, we write {\it ske} for short. Abusing notation, we shall also identify $\theta$ with the tuple of the images of the  
canonical generators of $\Delta.$
\subsection{Extending actions} Assume that $G'$ is a finite group such that $G \leqslant G'.$ The action of $G$ on $S$ represented by the ske \eqref{episs} is said to {\it extend} to an action of $G'$ on $S$ if:\begin{enumerate}
\item there is a Fuchsian group $\Delta'$ containing $\Delta,$ 
\item the Teichm\"{u}ller spaces of $\Delta$ and $\Delta'$ have the same dimension, and
\item there exists a ske  $$\Theta: \Delta' \to G' \, \, \mbox{ in such a way that }  \, \, \Theta|_{\Delta}=\theta \mbox{ and } \mbox{ker}(\theta)=\mbox{ker}(\Theta).$$
\end{enumerate} 

An action is called {\it maximal} if it cannot be extended in the previous sense. Singerman in \cite{singerman2} determined the complete list of pairs of signatures of Fuchsian groups $\Delta$ and $\Delta'$ for which it may be possible to have an extension as before. See also \cite{yoibero} and \cite{singerman1}.

\subsection{Equivalence of actions} \label{vecinos}
Two actions $\epsilon_i: G \to \mbox{Aut}(S)$  are  {\it topologically equivalent} if there exist $\omega \in \Aut(G)$ and an orientation preserving self-homeomorphism $f$ of $S$ such that
\begin{equation}\label{equivalentactions}
\epsilon_2(g) = f \epsilon_1(\omega(g)) f^{-1} \hspace{0.5 cm} \mbox{for all} \,\, g\in G.
\end{equation}

Each $f$ satisfying \eqref{equivalentactions} yields an automorphism $f^*$ of $\Delta$ where $\mathbb{H}/{\Delta} \cong S/G$. If $\mathscr{B}$ is the subgroup of $\mbox{Aut}(\Delta)$ consisting of them, then $\mbox{Aut}(G) \times \mathscr{B}$ acts on the set of skes defining actions of $G$ on $S$ with signature $\sigma(\Delta)$ by $$((\omega, f^*), \theta) \mapsto \omega \circ \theta \circ (f^*)^{-1}.$$  

Two skes $\theta_1, \theta_2 : \Delta \to G$ define topologically equivalent actions if and only if they belong to the same $(\mbox{Aut}(G) \times \mathscr{B})$-orbit; see, for example, \cite{Brou}. If the genus of $S/G$ is zero then $\mathscr{B}$ is generated by the so-called {\it braid transformations}  $\Phi_{i}$, for $1 \leqslant i  < l,$ defined by \begin{equation*} \label{braid} x_i \mapsto x_{i+1}, \hspace{0.3 cm}x_{i+1} \mapsto x_{i+1}^{-1}x_{i}x_{i+1} \hspace{0.3 cm} \mbox{ and }\hspace{0.3 cm} x_j \mapsto x_j \mbox{ when }j \neq i, i+1.\end{equation*}

\subsection{Equisymmetric stratification of $\mathscr{M}_g$} Following \cite{b}, the singular locus of  $\mathscr{M}_g$ admits an  equisymmetric stratification where 
each equisymmetric stratum, if nonempty, corresponds to one topological class of maximal actions (see also \cite{H}). More precisely:
$$\mbox{Sing}(\mathscr{M}_g)= \cup_{G, \theta} \bar{\mathscr{M}}_g^{G, \theta}$$where the {\it equisymmetric stratum} ${\mathscr{M}}_g^{G, \theta}$ consists of surfaces of genus $g$ with full automorphism group isomorphic to $G$ such that the action is topologically equivalent to $\theta$. In addition, the  {\it closure} $\bar{\mathscr{M}}_g^{G, \theta}$ of  ${\mathscr{M}}_g^{G, \theta}$ is a closed irreducible algebraic subvariety of $\mathscr{M}_g$ and consists of surfaces  of genus $g$ with a group of automorphisms isomorphic to $G$ such that the action is  topologically equivalent to $\theta$.

\s

 The subset $\bar{\mathcal{F}}_g(G, \sigma)=\bar{\mathcal{F}}_g$ of $\mathscr{M}_g$ of  all those  compact Riemann $S$ surfaces of genus $g$  with a group of automorphisms isomorphic to a given group $G$ acting with a given signature $\sigma$ will be called a {\it closed family}. Observe that if the signature of the action of $G$ on $S$ is \eqref{sig} then $$\dim (\bar{\mathcal{F}}_g)=3\gamma-3+s.$$

Assume that the action of $G$ is maximal. Then
\begin{enumerate}
\item the {\it interior} $\mathcal{F}_g$ of $\bar{\mathcal{F}}_g$ consists of those surfaces $S$ such that $G= \mbox{Aut}(S),$
\item $\mathcal{F}_g$ is formed by finitely many equisymmetric strata that are in correspondence with the pairwise non-equivalent topological actions of $G,$ and  \item  the set $\bar{\mathcal{F}}_g-\mathcal{F}_g$  is formed by  those surfaces $S$ such that $G <   \mbox{Aut}(S)$ properly. 

\end{enumerate}

\subsection{Abelian varieties}  A complex {\it abelian variety} is a complex torus which is also a complex projective algebraic variety. Each abelian variety $X=V/\Lambda$ admits a polarisation,  that is, a non-degenerate real alternating form $\Theta$ on $V$ such that for all $v,w \in V$$$\Theta(iv, iw)=\Theta(v,w) \,\,  \mbox{ and } \,\, \Theta(\Lambda \times \Lambda) \subset \mathbb{Z}.$$

If each elementary divisor of $\Theta|_{\Lambda \times \Lambda}$ equals 1 then
$\Theta$ is called {\it principal} and  $X$ is called a
 {\it principally polarised abelian variety}; from now on, we write  {\it ppav} for short. In this case,  there exists a basis for $\Lambda$ such that the matrix for $\Theta_{\Lambda \times \Lambda}$ with respect to it is given by \begin{equation}\label{simpl}
J = \left( \begin{smallmatrix}
0 & I_g \\
-I_g & 0
\end{smallmatrix} \right) \,\, \mbox{ where } \,\, g= \dim(X);
\end{equation}such a basis  is called {\it symplectic}. In addition, there exist a basis for $V$ with respect to which the period matrix for $X$ is $$\Pi=(I_g \, Z) \,\, \mbox{ where } \,\, Z  \in \mathscr{H}_g=\{ Z \in \mbox{M}(g, \mathbb{C}) : Z = Z^t , \, \mbox{Im}(Z) >0\},$$with $Z^t$ denoting the transpose matrix of $Z.$ The space $\mathscr{H}_g$ is called  Siegel upper half-space.

By an isomorphism of ppavs we mean an isomorphism of the underlying complex tori  preserving the involved polarisations. In other words, if $(I_{g} \, Z_i)$ is the period matrix of $X_i$ then an isomorphism $X_1 \to X_2$ is given by invertible matrices \begin{equation}\label{ig}M \in \mbox{GL}(g, \mathbb{C}) \,\, \mbox{ and } \,\, R \in \mbox{GL}(2g, \mathbb{Z}) \,\, \mbox{ such that } \,\,  M(I_{g} \, Z_1)=(I_{g} \, Z_2)R.\end{equation}

Since $R$ preserves the polarisation \eqref{simpl}, it belongs to the {\it symplectic group} $$\mbox{Sp}(2g, \mathbb{Z})=\{ R \in  \mbox{M}(2g, \mathbb{Z}) :  R^tJ R=J  \}. $$

It follows from \eqref{ig} that the correspondence $\mbox{Sp}(2g, \mathbb{Z}) \times \mathscr{H}_g \to \mathscr{H}_g$ given by
\begin{equation*}       (R=  \left( \begin{smallmatrix}
A & B \\
C & D
\end{smallmatrix} \right) , Z ) \mapsto R \cdot Z := (A+ZC)^{-1}(B+ZD)
\end{equation*}defines an action that identifies period matrices representing isomorphic ppavs. Hence $$\mathscr{H}_g \to \mathcal{A}_g:=\mathscr{H}_g/ \mbox{Sp}(2g, \mathbb{Z})$$is the moduli space of isomorphism classes of ppavs of dimension $g.$ See \cite{oort}.

\subsection{Abelian varieties with $G$-action} \label{ff2} Let $S$ be a compact Riemann surface of genus $g \geqslant 2.$ Consider the Jacobian variety $JS$ and its full (polarisation-preserving) automorphism group  $\mbox{Aut}(JS).$ Every automorphism of  $S$ induces a unique automorphism of $JS$. In fact $$[\mbox{Aut}(JS): \mbox{Aut}(S)] \in \{1,2\}$$according to whether or not $S$ is hyperelliptic; moreover, in the latter case $$ \mbox{Aut}(JS)/ \mbox{Aut}(S)=\{\pm  1\}.$$

As mentioned in the introduction, once a symplectic basis of $\Lambda=H_{1}(S, \mathbb{Z})$ is fixed, there is an isomorphism  $$\mbox{Aut}(JS) \cong \Sigma_S:=\{ R \in \Sp(2g, \mathbb{Z}) : R \cdot Z_S=Z_S\},$$where $(I_g \, Z_S)$ is the period matrix of $JS$. A change of basis induces a different but equivalent choice of $Z_S$ and a conjugate subgroup $\Sigma_S.$ One obtains a well-defined analytic submanifold$$\mathscr{S}_S:=\{Z\in \mathscr{H}_g: R \cdot Z=Z \text{ for all } R \in \Sigma_S\}$$of $\mathscr{H}_g$ whose points represent  ppavs admitting an action  equivalent to the one of $\mbox{Aut}(JS)$ in the symplectic group. Equivalently, as $- 1 \in \Sigma_S$, the previous submanifold represents ppavs admitting an action  equivalent to the one of $\mbox{Aut}(S).$   Clearly, $Z_S \in \mathscr{S}_S.$  

According to \cite{ST} (see also \cite[Lemma 3.8]{paola}), the dimension  $N_S$ of (the component which contains $JS$ of)  $\mathscr{S}_S$ agrees with   $$\dim (\mbox{Sym}^2 \mathscr{H}^{1,0}(S, \mathbb{C}))^{\tiny \mbox{Aut}(S)} $$where $\mbox{Sym}^2 \mathscr{H}^{1,0}(S, \mathbb{C})$ stands for the symmetric square of $\mathscr{H}^{1,0}(S, \mathbb{C})$. It follows that 
 $$N_S=\langle \chi_{\rho_a}^{\sy} | 1\rangle_G \,\, \mbox{ where } \,\, G=\mbox{Aut}(S)$$and $\chi_{\rho_a}^{\sy}$ denotes the character of the symmetric square of the analytic representation $\rho_a$ of $G$  and the brackets denote the usual inner product of characters of $G$.

\s

It is worth mentioning that $\mathscr{S}_S$ is related to some special subvarieties of $\mathcal{A}_g.$ Indeed,  as 
$\mbox{Aut}(JS)$ can be considered  as a subgroup of $$L_S:=\mbox{End}_{\mathbb{Q}}(JS)=\mbox{End}(JS)\otimes_{\mathbb{Z}} \mathbb{Q}$$one sees that $\mathscr{S}_S$ contains a complex submanifold  of $\mathscr{H}_g$ of matrices representing ppavs containing $L_S$ in their endomorphism algebras. This submanifold  is called a Shimura domain for $S$ and the corresponding ppavs form a so-called Shimura family for $S$; this a {\it special} subvariety of $\mathcal{A}_g$ (see \cite[\S3]{Moonen} for a precise definition). We refer to \cite[\S3]{wolfart} for more details.

\subsection{The group algebra decomposition}\label{jjo} The action of a group $G$ on a compact Riemann surface $S$  induces a $\mathbb{Q}$-algebra homomorphism from the rational group algebra of $G$  to $L_S$ $$\Xi : \mathbb{Q} [G] \to L_S.$$ 

Let $W_1, \ldots, W_r$ be the rational irreducible representations of $G,$ and for each $W_l$ let $V_l$ be a  complex irreducible representation of $G$ associated to it. Following \cite{l-r}, the equality \begin{equation}\label{noche1}1=e_1 + \cdots + e_r \,\, \mbox{ in }\,\, \mathbb{Q}[G],\end{equation}where $e_l$ is a  uniquely determined central idempotent associated to $W_l,$ yields an isogeny $$JS \sim A_{1} \times \cdots \times A_{r} \,\, \mbox{ where}\,\, A_{l} := \Xi (\alpha_l e_l)(JS)$$which is $G$-equivariant, with $\alpha_l \geqslant 1$ chosen to satisfy $\alpha_l e_l \in {\mathbb Z}[G]$. Additionally,  there are idempotents $f_{l1},\dots, f_{ln_l}$ such that \begin{equation}\label{noche2} e_l=f_{l1}+\dots +f_{ln_l}\end{equation}where  $n_l=d_{l}/s_{l}$ is the quotient of the degree $d_{l}$ and the Schur index $s_{l}$ of $V_l$. These idempotents provide $n_l$ pairwise isogenous subvarieties of $JS.$ If we denote by  $B_l$  one of them for each $l,$ then \eqref{noche1} and \eqref{noche2} provide the isogeny
\begin{equation} \label{eq:gadec}
JS \sim B_{1}^{n_1} \times \cdots \times B_{r}^{n_r} 
\end{equation}
known as the {\it group algebra decomposition} of $JS$ with respect to $G$. See  \cite{cr}.

\s

Let $H$ be a subgroup of $G.$  We  denote by $d_{l}^H$  the dimension of the vector subspace of $V_l$  of those elements which are fixed under $H.$ Following \cite[Proposition 5.2]{cr},   the group algebra decomposition   \eqref{eq:gadec} induces the following isogeny of the Jacobian $J(S/H)$ of the quotient $S/H$\begin{equation} \label{decoind1}
J(S/H) \sim  B_{1}^{{n}_1^H} \times \cdots \times B_{r}^{n_r^H} \,\,\, \mbox{ where } \,\,\, {n}_l^H=d_{l}^H/s_{l}.
\end{equation}

The previous  isogeny has proved to be fruitful in finding Jacobians $JS$ isogenous to a product of Jacobians of quotients of $S.$ See, for example,  \cite{kanirubiyo} and also \cite{RCR}.

\s

Assume that $(\gamma; k_1, \ldots, k_s)$ is the signature of the action of $G$ on $S$  and that this action is represented by the ske $\theta: \Delta \to G,$ with $\Delta$ as in \eqref{prese}. Observe that if $V_1=W_1$ denotes the trivial representation of $G$ then $B_1 \sim J(S/G)$ and therefore $\dim B_1 = \gamma.$ If $l \geqslant 2$ then, according to \cite[Theorem 5.12]{yoibero}, we have that
\begin{equation}\label{dimensiones1}
\dim B_{l}=m_{l}[d_{l}(\gamma -1)+\tfrac{1}{2}\Sigma_{j=1}^s (d_{l}-d_{l}^{\langle \theta(x_j) \rangle} )]  \end{equation} where $m_{l}$ is the degree of  $\mathbb{Q} \le L_{l}$ with $L_{l}$ denoting a minimal field of definition for $V_l.$

\s

For decompositions of Jacobians and families of Jacobians with respect to special groups, we refer to the articles \cite{Ba}, \cite{d1}, \cite{CRC}, \cite{Do}, \cite{FPa},  \cite{nos},  \cite{PA}, \cite{d3}, \cite{yosd},  \cite{yojpaa} and \cite{Ri}.

\section{Proof of Theorem \ref{tm}} \label{ptm} The proof of Theorem \ref{tm} is presented as a consequence of a series of propositions proved in this section. Hereafter, we  assume $q \geqslant 7$ to be prime and $S$ to be a  Riemann surface of genus $g:=q-1$  with a group of automorphisms $G$ of order $\lambda q$ where $\lambda \geqslant 1$ is an integer. 

\begin{propn} \label{lambda3}
If $\lambda=3$ then  $G$ is cyclic and acts with signature  $(0; 3,q,3q).$ Moreover,  $S$
is unique up to isomorphism  and $G$ is its full automorphism group.\end{propn}

\begin{proof} Let $(\gamma; k_1, \ldots, k_l)$ be the signature of the action of $G$ on $S.$  The Riemann-Hurwitz formula implies that \begin{equation}\label{bus}2(q-2) \geqslant 3q(2\gamma-2+\tfrac{2}{3}l).\end{equation}

Observe that if $\gamma \geqslant 1$ then $l=0$ and therefore $q=2,$ contradicting the assumption $q \geqslant 7$. We then assume $\gamma=0$ and therefore \eqref{bus} shows that $l=3.$ It follows that the signature of the action of $G$  is $$(0; k_1, k_2, k_3) \,\, \mbox{ where }\,\, \tfrac{1}{k_1}+\tfrac{1}{k_2}+\tfrac{1}{k_3}=\tfrac{1}{3}+\tfrac{4}{3q} \,\, \mbox{ and }\,\, k_j \in \{3,q,3q\}.$$ After a routine computation, one sees that the unique solution of the previous equation is, up to permutation, $k_1=3, k_2=q$ and $k_3=3q.$ The last equality  implies that $$G \cong C_q \times C_3 = \langle \alpha, \beta : \alpha^q=\beta^3=1, [\alpha, \beta]=1 \rangle.$$

Consider the Fuchsian group $\Delta$  of signature $(0; 3,q,3q)$ canonically presented   \begin{equation*} \label{luci} \Delta=\langle w_1, w_2, w_3 : w_1^3=w_2^q=w_3^{3q}=w_1w_2w_3=1\rangle\end{equation*}and let $\theta: \Delta \to G$ be a ske representing an action of $G$ on $S.$ It is not difficult to see that, up to an automorphism of $G$, the ske $\theta$ is given by $$\theta(w_1)=\beta, \,\, \theta(w_2)=\alpha \,\, \mbox{ and }\,\, \theta(w_3)=\alpha^{-1}\beta^2;$$this proves the uniqueness of $S$. By the results of \cite{singerman2}, if $G$ is strictly contained in the full automorphism group $\mbox{Aut}(S)$ of $S$  then $\mbox{Aut}(S)$ has order $12q$, acts on $S$ with signature $(0; 2,3,3q)$ and $G$ is a non-normal subgroup of it. By the classical Sylow's theorem,  if a group of order $12q$ with $q > 11$ has a non-normal subgroup isomorphic to $G$ then it is isomorphic to $C_q \times A_4$ where $A_4$ stands for the alternating group of order 12. However, the product of an element of order two and an element of order three of $C_q \times A_4$ cannot have order $3q.$ The cases $q \leqslant 11$ are not realised either; see \cite{conder}. The proof is done.
\end{proof}

\begin{propn} \label{567} $\lambda$ is different from $5,6$ and $7.$
\end{propn}

\begin{proof} If $\lambda$ equals $5, 6$ or $7$ then $G$ is a large group of automorphisms (that is, $ |G| >4(g-1)$) and therefore (see, for example, \cite[\S2.3]{K1}) 
 the signature of the action is either
 \begin{enumerate} \item $(0; k_1, k_2, k_3)$ for some $2 \leqslant k_1 \leqslant k_2 \leqslant k_3,$  
 \item $(0; 2,2,2,k)$ for some $k \geqslant 3,$ or
 \item $(0; 2,2,3,k)$ for some $3 \leqslant k \leqslant 5.$ 
\end{enumerate}

If $\lambda$ equals 5 or $7$ then $G$ has no involutions;  then the signature of the action $G$ is $$(0; k_1, k_2, k_3) \,\, \mbox{ for some }\,\, k_j \in \{5,q,5q\} \, \mbox{ or } \,\,  k_j \in \{7,q,7q\}$$respectively. The Riemann-Hurwitz formula implies that \begin{equation*}\label{grama}\tfrac{1}{k_1}+\tfrac{1}{k_2}+\tfrac{1}{k_3}=\tfrac{3}{5}+\tfrac{4}{5q} \,\, \mbox{ and }\,\, \tfrac{1}{k_1}+\tfrac{1}{k_2}+\tfrac{1}{k_3}=\tfrac{5}{7}+\tfrac{4}{7q}\end{equation*}respectively and this, in turn, implies that $q$ is negative; a contradiction.

We now assume that $G$ has order $6q$. If the signature of the action of $G$ is $(0; 2,2,2,k)$ then, by the Riemann-Hurwitz formula, we have  that $q+4$ divides $6q$ and therefore $q = 2;$  contradicting the assumption $q \geqslant 7$. The signatures $(0; 2,2,3,4)$ and $(0; 2,2,3,5)$ cannot be realised either, since a group of order $6q$ does not have elements of order $4$ nor $5.$ Besides, a direct computation shows that the signature   $(0; 2,2,3,3)$ contradicts the Riemann-Hurwitz formula.

 It follows that the signature of the action is $(0; k_1, k_2, k_3)$ where $k_j \in \{2,3,6,q,2q,3q, 6q\}$ satisfy, by the Riemann-Hurwitz formula, the equality $$\tfrac{1}{k_1}+\tfrac{1}{k_2}+\tfrac{1}{k_3}=\tfrac{2}{3}+\tfrac{2}{3q}.$$

Set $v=\#\{k_j : k_j=3\}.$ It is clear that  $v \neq 2,3$. Assume $v=1$ and say $k_1=3$. If $k_2, k_3 \geqslant 6$ then $q \leqslant 0.$ Then, we can assume $k_2=2$ and therefore $$\tfrac{1}{k_3}=-\tfrac{1}{6}+\tfrac{2}{3q} \,\, \mbox{ showing that }\,\, q \leqslant 3.$$Thus, $v=0.$ Now, let $u=\#\{k_j : k_j=2\}$ and observe that $u \neq 2,3.$ If $u=1$ then $$\tfrac{1}{k_2}+\tfrac{1}{k_3}=\tfrac{2}{3q}-\tfrac{1}{6} \,\, \mbox{ and therefore }\,\, q \leqslant 3.$$

All the above ensures that each $k_j \geqslant 6.$ It follows that $$\tfrac{1}{k_1}+\tfrac{1}{k_2}+\tfrac{1}{k_3}=\tfrac{2}{3}+\tfrac{2}{3q}\leqslant \tfrac{1}{2}$$and therefore $q < 0;$ contradicting the assumption $q \geqslant 7$.
\end{proof}

\begin{propn} \label{tutu}
If $\lambda \geqslant 8$ then $\lambda=8$ and $S \cong X_8.$
\end{propn}

\begin{proof} If the order of $G$ is at least $8(g+1)$ then,  following \cite[p. 77]{accola},  the signature of the action of $G$ is either$$\begin{array}{lcl}
(1) \,\, (0; 2,2,2,3), & \,\,\,\,\,  & (5) \,\, (0; 2,6,k) \mbox{ where } 6 \leqslant k \leqslant 11, \\
(2) \,\, (0; 2,3,k) \mbox{ where } k \geqslant 7, & \,\,\,\,\, & (6)\,\,(0; 2,7,k) \mbox{ where } 7 \leqslant k \leqslant 9, \\
(3) \,\, (0; 2,4,k) \mbox{ where } k \geqslant 5, & \,\,\,\,\, & (7) \,\, (0; 3,3,k) \mbox{ where } 4 \leqslant k \leqslant 11, \mbox{ or}\\
(4) \,\, (0; 2,5,k) \mbox{ where } 5 \leqslant k \leqslant 19, & \,\,\,\,\, & (8) \,\, (0; 3,4,k)\mbox{ where } 4 \leqslant k \leqslant 5.
\end{array}$$

\s

We observe that cases (1), (5) and (8) are not realised. Indeed, this fact follows from the contradiction between the fourth and fifth columns in  the following table.

\s

\begin{center}
\begin{tabular}{|c|c|c|c|c|}  
\hline

\, case \, & \, signature \, &  $|G|$  &  \, condition \, & Riemann-Hurwitz formula   \\ [0.4ex]\hline 
\, (1) \, & $(0; 2,2,2,3)$  & $6\lambda'q$ \, & $\lambda' \geqslant 2$ & $\lambda'=1-\tfrac{2}{q}$ \\ [0.4ex]
\, (5) \, & $(0; 2,6,k)$ &  $6 \lambda'q$ \, & $\lambda' \geqslant 2$ & $\lambda'=\tfrac{k}{k-3}(1-\tfrac{2}{q})$  \\[0.4ex]  \, (8.1) \, & 
$(0; 3,4,4)$  & $12\lambda'q $ \, & $ \lambda' \geqslant 1$ & $\lambda'=(1-\tfrac{2}{q})$ \\[0.4ex] \, (8.2) \, & 
$(0; 3,4,5)$  & $60\lambda'q $ \, & $\lambda' \geqslant 1$ & $\lambda'=\tfrac{2}{13}(1-\tfrac{2}{q})$ \\[0.4ex] 
\hline
\end{tabular}
\end{center}

\s

We also note that  cases (4), (6) and (7) are not realised. Indeed

\s

\begin{center}
\begin{tabular}{|c|c|c|c|c|}  
\hline

\, case \, & \, signature \, &  $|G|$  & \, condition \, & Riemann-Hurwitz formula   \\ [0.4ex]\hline \, (4) \, &
$(0; 2,5,k)$  & $10\lambda'q $ \, & $\lambda' \geqslant 1$ & $\lambda'=\tfrac{2k}{3k-10}(1-\tfrac{2}{q})$ \\ [0.4ex] \, (6) \, &
$(0; 2,7,k)$ &  $2 \lambda'q $ \, & $\lambda' \geqslant 4$ & $\lambda'=\tfrac{14k}{5k-14}(1-\tfrac{2}{q})$  \\[0.4ex] 
\, (7) \, & $(0; 3,3,k)$  & $3\lambda'q $ \, & $\lambda' \geqslant 3$ & $\lambda'=\tfrac{2k}{k-3}(1-\tfrac{2}{q}) $ \\[0.4ex] 
\hline
\end{tabular}
\end{center}

\s
\s
and notice that:
\begin{enumerate}
\item[(a)] in  case (4) we have that $\lambda'=1$ and therefore $q=4k/(10-k)$ and $5 \leqslant k \leqslant 9.$ However, for each $k$ as before we obtain that $q$ is not prime;
\item[(b)] in  case (6) we have that $\lambda'=4$ and therefore $q$ equals $14$, $28$ and $126$ for $k=7, 8$ and $9$ respectively; and
\item[(c)] in case (7) the facts that $\lambda \geqslant 3$ and $q \geqslant 7$ imply that $k \leqslant 8.$ If $k=4$ then $q$ is not prime, if $k=5$ then $\lambda'=3$ or $4$ and $q=5$ or 10, and if $k=6,7,8$ then $\lambda'=3$ and $q$ is not prime.
\end{enumerate}

\s

We claim that  case (2) is not realised either. Indeed, note that otherwise the order of $G$ equals $6\lambda'q$ for some $\lambda' \geqslant 2$ and the Riemann-Hurwitz formula reads $$k=\tfrac{6q\lambda'}{q(\lambda'-2)+4} \,\, \mbox{ and therefore }\,\,k':=\tfrac{6\lambda'}{q(\lambda'-2)+4} \in \mathbb{Z}^+$$

\begin{enumerate}
\item If $k'=1$ then $\lambda'=2+8/(q-6),$ showing that $q=7$ and $\lambda'=10.$ Consequently, the order of $G$ is $420$ and acts on $S$ of genus six with signature $(0; 2,3,7).$ However, such a Riemann surface does not exist because the maximal number of automorphisms that a Riemann surface of genus six can admit is 150 (see, for instance, \cite{conder}).
\s

\item If $k' \geqslant 2$ then  $\lambda' \leqslant 2 + 2/(q-3)$ and therefore $\lambda'=2.$ It follows that $G$ has order $12q$ and acts on $S$ with signature $(0; 2,3,3q).$ Observe that the signature of the action shows, in particular, that $S$ has a cyclic subgroup $H < G$ of automorphisms of order $3q$. However, as proved in Proposition \ref{lambda3}, if $S$ has a group of automorphisms of order $3q$ then $S$ does not have more  automorphisms; a contradiction.

\end{enumerate}
This proves the claim.
\s

All the above ensures that the signature of the action is  $(0; 2,4,k)$ for some $k \geqslant 5.$ Observe that the order of $G$ is $4 \lambda'q$ for some $\lambda' \geqslant 2$ and the Riemann-Hurwitz formula says $$\lambda'=\tfrac{2k}{k-4}(1-\tfrac{2}{q}) < \tfrac{2k}{k-4}.$$It follows that one of the following statements holds. \begin{enumerate}
\item $k=5$ and $\lambda' \in \{3, \ldots, 9\}$ and therefore $q=20/(10-\lambda').$
\item $k=6$ and $\lambda' \in \{3, 4, 5\}$ and therefore $q=12/(6-\lambda').$ 
\item $k\geqslant 5$ and $\lambda'=2,$ and therefore $q=k/2.$ 
\end{enumerate}

The first two cases must be disregarded because $q$ is not prime; then $G$ has order $8q$ and acts with signature $(0; 2,4,2q)$. By \cite[\S5]{K1}, we obtain that $S \cong X_8$ as desired.
\end{proof}

We recall that, following \cite{CI}, the closed  family $\bar{\mathscr{C}}_g$ consists of all
 those  compact Riemann surfaces  of genus $g$ endowed with a group of automorphisms $G$ isomorphic to $$\mathbf{D}_{q} \times C_2 \cong \mathbf{D}_{2q}$$ acting with signature $(0; 2,2,2,q).$ Moreover, if $S$ belongs to the interior of $\bar{\mathscr{C}}_g$ then $G$ agrees with the full automorphism group of $S.$ It was also observed in \cite{CI} that $X_8 \in \bar{\mathscr{C}}_g-{\mathscr{C}}_g.$
\begin{propn} \label{valp}
$\bar{\mathscr{C}}_g - \mathscr{C}_g=\{X_8\}.$
\end{propn}

\begin{proof}Observe that if $X$ belongs to $\bar{\mathscr{C}}_g - \mathscr{C}_g$ then its automorphism group has order $4tq$ for some $t \geqslant 2.$ It follows from Proposition \ref{tutu} that $t=2$ and that  $X\cong X_8.$
\end{proof}

For later and repeated use, we recall here that \begin{equation} \label{grupoAM}\mbox{Aut}(X_8) \cong \langle x,y,z: x^{2q}=y^2=z^2=1, [x,y]=[z,y]=1, zxz=x^{-1}y \rangle\end{equation}and its action on $X_8$ is represented by the ske \begin{equation} \label{epiAM}\Theta: \Delta_8 \to \mbox{Aut}(X_8) \mbox{ given by } (z_1,z_2, z_3) \to (z,zx,x^{-1}),\end{equation}where $\Delta_8$ is a Fuchsian group of signature $(0; 2,4,2q)$  presented as 
 \begin{equation} \label{fuAM}\Delta_8=\langle z_1, z_2, z_3 : z_1^2=z_2^4=z_3^{2q}=z_1z_2z_3=1
\rangle.\end{equation}

\begin{propn} \label{verde}
If $X$ is a compact Riemann surface of genus $g$ with a group of automorphisms isomorphic to $C_q \times C_2^2$ acting with signature $(0; 2,2q,2q)$ then $X \cong X_8.$
\end{propn}

\begin{proof} Let $\Delta_2$ be a  Fuchsian group of signature $(0; 2,2q,2q)$ presented as \begin{equation} \label{med}\Delta_2= \langle y_1, y_2, y_3 : y_1^2=y_2^{2q}=y_3^{2q}=y_1y_2y_3=1 \rangle\end{equation}and consider the group $G \cong C_q \times C_2^2$ presented as $$\langle A, B, C : A^q=B^2=C^2=(BC)^2=[A,B]=[A,C]=1 \rangle.$$

Let $\theta : \Delta_2 \to G$ be a ske representing the action of $G$ on $X$. Up to an automorphism of $G$ we can assume $\theta(y_1)=B.$ Moreover, after considering the automorphism of $G$ given by $$A  \mapsto A, \,\, B \mapsto B, \,\,C \mapsto BC,$$ we can assume that $\theta(y_2)$ equals  either $A^iB$ or $A^iC$ for some $i \in \mathbb{Z}_q^*.$ Note that the former case is impossible, since $\theta(y_1y_2)$ would not have order $2q.$ Thus, after sending $A$ to an appropriate power of it, we obtain that $\theta$ is equivalent to  \begin{equation} \label{raton} \Delta_2 \to C_q \times C_2^2 \,\, \mbox{ given by } (y_1, y_2, y_3) \mapsto (B,AC,(ABC)^{-1}).\end{equation}

Observe that, with the notations of \eqref{fuAM}, the elements $$\hat{y}_1:=z_2^{-2}, \,\, \hat{y}_2:=z_3^{-1} \,\, \mbox{ and }\,\, \hat{y}_3=z_2^{-1}z_3^{-1}z_2$$generate a  subgroup of $\Delta_8$ isomorphic to $\Delta_2$ and the restriction of \eqref{epiAM} to it 
\begin{equation} \label{corona}\Delta_2 \to C_q \times C_2^2 \,\, \mbox{ is given by } \,\, (\hat{y}_1, \hat{y}_2,\hat{y}_3 ) \mapsto (y,x,x^{-1}y).\end{equation} 
By letting $x=AC$
 and $y=B$ we see that \eqref{raton} and \eqref{corona} agree; consequently $X \cong X_8$. 
\end{proof}

\begin{propn} \label{carga}
If $Y$ is a compact Riemann surface of genus $g$ with a group of automorphisms isomorphic to $C_q \rtimes_2 C_4$ acting with signature $(0; 4,4,q)$ then $Y \cong X_8.$
\end{propn}

\begin{proof} Let $\Delta_3$ be a  Fuchsian group of signature $(0; 4,4,q)$ presented as \begin{equation} \label{jcc}\Delta_3= \langle y_1, y_2, y_3 : y_1^4=y_2^{4}=y_3^{q}=y_1y_2y_3=1 \rangle\end{equation}and consider the group $G \cong C_q \rtimes_2 C_4$ presented as $$\langle A, B : A^q =B^4=1, BAB^{-1}=A^{-1}\rangle.$$ 
 
Let $\theta : \Delta_3 \to G$ be a ske representing the action of $G$ on $Y$. Then, after sending $B$ to $B^{-1}$ if necessary, $\theta$ is given by $$(y_1, y_2, y_3) \mapsto (A^iB, A^jB^{-1}, A^k) \,\, \mbox{ for some } i,j\in \mathbb{Z}_q \mbox{ and } k \in \mathbb{Z}_q^*.$$Up to   conjugation, we can assume $i=0$ and after sending $A$ to an appropriate power of it, we can assume $k=1.$ It follows that $j=1$ and therefore $\theta$ is equivalent to \begin{equation} \label{ence}  \Delta_3 \to C_q \rtimes_2 C_4 \,\, \mbox{given by} \,\, (y_1, y_2, y_3) \mapsto (B, AB^{-1}, A).\end{equation} Now, as done in the previous proposition, with the notations of \eqref{fuAM}, we see that $$\tilde{y}_1:=z_2, \,\, \tilde{y}_2:=z_3z_2z_3^{-1} \,\, \mbox{ and }\,\, \tilde{y}_3=z_3^{2}$$generate a  subgroup of $\Delta_8$ isomorphic to $\Delta_3$ and the restriction of \eqref{epiAM} to it \begin{equation} \label{corona2}\Delta_3 \to C_q \rtimes_2 C_4 \,\, \mbox{ is given by } \,\, (\tilde{y}_1, \tilde{y}_2,\tilde{y}_3 ) \mapsto (zx, x^{-3}z, x^{-2}).\end{equation}Write $A=x^{-2}$ and $B=zx$ to see that \eqref{ence} and \eqref{corona2} agree; consequently $Y\cong X_8.$
\end{proof}

\begin{propn} \label{cargados} Assume $q \equiv 1 \mbox{ mod } 4.$
There exists a unique, up to isomorphism, compact Riemann surface $X_4$ of genus $g$ with full automorphism group isomorphic to $C_{q} \rtimes_4 C_4$ acting on it  with signature $(0; 4,4,q).$
\end{propn}

\begin{proof}
Consider the Fuchsian group $\Delta_3$ as in \eqref{jcc}, and  the group\begin{equation}\label{gratis}G \cong C_q \rtimes_4 C_4=\langle A, B : A^q=B^4=1, BAB^{-1}=A^\rho \rangle \end{equation}where $\rho$ is a primitive fourth root of unity in $\mathbb{Z}_q.$ If $\theta : \Delta_3 \to C_q \rtimes_4 C_4$ is a ske representing the action of $G$ on a compact Riemann surface $Z$ of genus $g$ then, by proceeding similarly as done in the previous proposition, one sees that $\theta$ is equivalent to  $$\theta_1(y_1, y_2, y_3) = (A^{-1}B, B^{-1}, A) \,\, \mbox{ or }\,\, \theta_2(y_1, y_2, y_3) = (A^{-1}B^{-1}, B, A).$$

It follows that, up to isomorphism, there are at most two  surfaces $Z$ as before; namely $$Z_j :=\mathbb{H}/K_j \,\, \mbox{ where } \,\, K_j=\mbox{ker}(\theta_j) \, \, \mbox{ for } j=1,2.$$

Observe that if the full automorphism  group of
 $Z_j$ is different from $G$ then, by Proposition \ref{tutu}, necessarily $Z_j \cong X_8$ and, in particular, $\mbox{Aut}(X_8)$ contains a subgroup isomorphic to $C_q \rtimes_4 C_4.$ However, this is not possible.  Indeed, with the notations of \eqref{grupoAM}, if $\iota \in \mbox{Aut}(X_8)$ has order 4 then $\iota^2$ equals the central element $y.$  It follows that $\mbox{Aut}(Z_j) \cong G$ for $j=1,2.$ 
\s

We record here that $Z_1$ and $Z_2$ are isomorphic if and only if $K_1$ and $K_2$ are conjugate in $\mbox{Aut}(\mathbb{H})$. As the normaliser of each $K_j$ is  $\Delta_3,$ it can be seen that $K_1$ and $K_2$ are conjugate  if and only if they are conjugate in the normaliser $N(\Delta_3)$ of $\Delta_3.$ Now, the action by conjugation of $N(\Delta_3)$ on $\{K_1, K_2\}$  has orbits of length $[N(\Delta_3): \Delta_3]$ which is, by \cite[Theorem 1]{singerman2}, equal to 2. Hence,  $K_1$ and $K_2$ are conjugate and thus  $X_4:=Z_1 \cong Z_2$ as desired. 
\end{proof}

\begin{propn} \label{agro} If $\lambda=4$ and $S \notin \bar{\mathscr{C}}_g$ then $q \equiv 1 \mbox{ mod } 4$ and $S \cong X_4.$ 
\end{propn}

\begin{proof} As in the proof of Proposition \ref{567}, 
 the signature of the action of $G$ on $S$  is either \begin{enumerate} \item $(0; k_1, k_2, k_3)$ for some $2 \leqslant k_1 \leqslant k_2 \leqslant k_3,$    
 \item $(0; 2,2,2,k)$ for some $k \geqslant 3,$ or
 \item $(0; 2,2,3,k)$ for some $3 \leqslant k \leqslant 5.$ 
\end{enumerate}

The third case must be disregarded since there is no group of order $4q$ with an element of order three. If the signature is as in the second case, then the Riemann-Hurwitz formula implies that $k=q.$ We now assume the signature to be as in the first case. The Riemann-Hurwitz formula says \begin{equation} \label{virus}\tfrac{1}{k_1}+\tfrac{1}{k_2}+\tfrac{1}{k_3}=\tfrac{1}{2}+\tfrac{1}{q}.\end{equation} Note that among $k_1, k_2, k_3$ not  two or three of them can be equal to 2. It follows that, up to permutation, there are two cases to consider.
\begin{enumerate}

\item Assume $k_1=2$ and $k_2, k_3  \geqslant 4.$ Then \eqref{virus} turns into $\tfrac{1}{k_2}+\tfrac{1}{k_3}=\tfrac{1}{q}.$ Note that if $k_2=4$ then $k_3 \leqslant 0.$ It follows that $k_2, k_3 \geqslant q$ and therefore $k_2=k_3=2q.$  

\s

\item Assume $k_1, k_2, k_3 \geqslant 4.$ If the number of periods $k_j$ that are equal to 4 is 2 then \eqref{virus} implies that the signature is $(0; 4,4,q);$ otherwise $q \leqslant 4.$
\end{enumerate}

\s

Thereby, the signature of the action of $G$ on $S$ is either $$(0; 2,2,2,q), \,\, (0; 2, 2q, 2q) \,\, \mbox{ or }\,\, (0; 4,4,q).$$

We recall that if $q \equiv 3 \mbox{ mod } 4$ then $G$ is isomorphic to either ${C}_{4q}$, ${C}_q \times {C}_2^2,$ $\mathbf{D}_{2q}$ or $C_q \rtimes_2 C_4,$ and if $q \equiv 1 \mbox{ mod } 4$ then, in addition, $G$ can be isomorphic to $C_q \rtimes_4 C_4.$ 

\s

\begin{enumerate}
\item If $G$ acts with  signature $(0; 2,2,2,q)$ then  $G$ is generated by three involutions and therefore $G \cong \mathbf{D}_{2q},$ showing that  $S \in \bar{\mathscr{C}}_g.$

\s

\item If $G$ acts with signature $(0; 2,2q,2q)$ then   $G$ is generated by two elements of order $2q$ whose product is an involution; thus $G \cong C_q \times C_2^2.$ By Proposition \ref{verde}, we see that $S \cong X_8$ and therefore $S \in \bar{\mathscr{C}}_g.$

\s

\item If $G$ acts with signature  $(0; 4,4,q)$ then $G$ is generated by two elements of order 4 and therefore $G$ is isomorphic to  $C_q \rtimes_2 C_4$ or $C_q \rtimes_4 C_4.$ By Proposition \ref{carga} the former case implies $S \cong X_8$ whilst  by Proposition \ref{cargados} the latter case implies $S \cong X_4.$
\end{enumerate}This finishes the proof.
\end{proof}

 \begin{propn} \label{lapiz}
 If $\lambda=2$ and $G$ is cyclic acting with signature $(0; 2,2,q,q)$ then $S \in \bar{\mathscr{C}}_g.$
 \end{propn}
 
 \begin{proof}
 
 Let $\Delta_4$ be a Fuchsian group of signature $(0; 2,2,q,q)$  presented as \begin{equation} \label{cuatro}\Delta_4=\langle x_1, x_2, x_3, x_4 : x_1^2=x_2^2=x_3^q=x_4^q=x_1x_2x_3x_4=1 \rangle\end{equation} and consider the cyclic group of order $2q$ generated by $a$ of order $q$ and $b$ of order two.  As $G$ has only one involution, it is clear that, after sending $a$ to a suitable power of it, each ske representing an action of $G$ on $S$  is equivalent to \begin{equation}\label{tetacero} \theta_0 : \Delta_4 \to C_q \times C_2 \,\, \mbox{ such that }\,\, \theta_0(x_1, x_2, x_3, x_4)=(b,b,a,a^{-1}).\end{equation}Then, such  surfaces $S$ form an equisymmetric complex one-dimensional family.  Let \begin{equation} \label{g44}\Delta_1=\langle y_1, y_2, y_3, y_4 :  y_1^2= y_2^2= y_3^2= y_4^q= y_1 y_2 y_3 y_4=1\rangle\end{equation} be a Fuchsian group of signature $(0; 2,2,2,q)$ and consider the group\begin{equation} \label{frio} \mathbf{D}_{2q}=\langle R,T : R^{2q}=T^2=(TR)^2=1 \rangle.\end{equation}

 We recall that, following \cite{CI}, the action of $ \mathbf{D}_{2q}$ on $S' \in \bar{\mathscr{C}}_g$ is represented by the ske \begin{equation} \label{antes}\theta: \Delta_1 \to \mathbf{D}_{2q} \,\, \mbox{ given by } \,\, (y_1, y_2, y_3, y_4) \mapsto ( R^q, T, TR, R^{q-1})\end{equation} Now, the elements of $\Delta_1$ $$\hat{x}_1:= y_1, \,\, \hat{x}_2:=y_2y_1y_2, \,\, \hat{x}_3:= y_4 \,\, \mbox{ and }\,\, \hat{x}_4:= y_1y_2y_4y_2y_1 $$generate a Fuchsian group isomorphic to $\Delta_4.$ The restriction of \eqref{antes} to it  \begin{equation}\label{ciber}\Delta_4 \to \langle R \rangle \cong G\,\, \mbox{ is given by  } \,\, (\hat{x}_1, \hat{x}_2, \hat{x}_3, \hat{x}_4) \mapsto (R^q,R^q,R^{q-1}, R^{1-q}).\end{equation}Set $a:=R^{q-1}$ and $b:=R^{q}$ to see that \eqref{ciber} agrees with \eqref{tetacero} and the result follows.\end{proof}

\begin{propn}
\label{lapiz2}
 If $\lambda=2$ and $G$ is a cyclic group acting with signature $(0; q,2q,2q),$ then  either 
 $S \cong X_8$ or 
 the full automorphism group of $S$ agrees with $G.$
In the latter case, 
there are exactly $\tfrac{q-3}{2}$ pairwise non-isomorphic compact Riemann surfaces.
\end{propn}

\begin{proof}
Let $\Delta_5$ be a Fuchsian group of signature $(0; q,2q,2q)$  presented as \begin{equation*}\label{azul8}\Delta_5=\langle x_1, x_2, x_3 : x_1^q=x_2^{2q}=x_3^{2q}=x_1x_2x_3=1 \rangle\end{equation*} and  consider the cyclic group of order $2q$ generated by $a$ of order $q$ and $b$ of order two.

If $\theta: \Delta_5 \to G $ is a ske representing the action of $G$ on $S$ then after sending $a$ to an appropriate power of it, we see that $\theta$ is equivalent to $$\theta_j = (a,a^jb, a^{-j-1}b) \mbox{ for some } \, j \neq -1,0.$$

Let $S_j$ be the compact Riemann surface defined by $\theta_j$ and write $j^*=\tfrac{q-1}{2}.$ 

\s

We claim that $S_{j^*} \cong X_8.$ To prove that, we notice that, by Proposition \ref{verde}, it suffices to verify that $\theta_{j^*}$ extends to the action of $C_q \times C_2^2$ with signature 
$(0; 2,2q,2q).$ With the notations of the proof of Proposition \ref{verde}, the elements $$\hat{x}_1:= y_3^2, \,\, \hat{x}_2:=y_1y_2y_1  \,\, \mbox{ and }\,\, \hat{x}_3:= y_2 $$ generate a subgroup of \eqref{med} isomorphic to $\Delta_5$ and the restriction of \eqref{raton} to it \begin{equation}\label{look}\Delta_5 \to \langle A, C\rangle \cong G\,\, \mbox{ is given by  } \,\, (\hat{x}_1, \hat{x}_2, \hat{x}_3) \mapsto (A^{-2}, AC, AC).\end{equation}By setting $a:=A^{-2}$ and $b:=C,$ we see that \eqref{look} is equivalent to ${\theta}_{j^*},$ as desired. 

\s

Let $j \neq j^*.$  If $\mbox{Aut}(S_j) \neq G$ then, by \cite{singerman2} and Proposition \ref{tutu}, the action $\theta_j$ must extend to the action  \eqref{raton} of $C_q \times C_2^2$ with signature $(0; 2,2q,2q).$
Observe that an element $y$ of \eqref{med} has order $2q$ if and only if it is conjugate to $y_2^k$ or to $y_3^{k}$ for some $k \in \{1, \ldots, 2q-1\}$  odd and different from $q.$ As the target group is abelian, the image of $y$ under \eqref{raton} is either $AC$ or $(ABC)^{-1}$. Now,  if $\Delta'$ is a subgroup of \eqref{med} isomorphic to $\Delta_5$ then the restriction of \eqref{raton} to the canonical generators of $\Delta'$  must be $$(A^{-2}, AC, AC), \,\,(B,AC, (ABC)^{-1}) \, \, \mbox{ or }\,\,(A^2,(ABC)^{-1},(ABC)^{-1}).$$

The second case is impossible since it does not have the required signature; the other two cases are equivalent to $\theta_{j^*}$. We conclude that if $j \neq j^*$ then  $\mbox{Aut}(S_j)=G$ and, in particular, $S_j$ is not isomorphic to $S_{j^*}$.

\s

Write $K_j=\mbox{ker}(\theta_j)$ for each $j \in \mathbb{Z}_q-\{-1,0,j^*\}.$ As argued in the proof of Proposition \ref{cargados}, we have that $S_{j_1}$ and $S_{j_2}$ are isomorphic if and only if $K_{j_1}$ and $K_{j_2}$ are conjugate in the normaliser $N(\Delta_5)$ of $\Delta_5.$ The action by conjugation of $N(\Delta_5)$ on $$\{K_j : j \in \mathbb{Z}_q-\{-1,0,j^*\}\}$$has orbits of length $[N(\Delta_5): \Delta_5]$ which is, by \cite[Theorem 1]{singerman2}, equal to 2. Hence,$$\{S_j : j \in \mathbb{Z}_q-\{-1,0,j^*\}\}$$splits into $\tfrac{q-3}{2}$ isomorphism classes. Finally, observe that the elements $x_2$ and $x_3$ of $\Delta_5$ are conjugate in $N(\Delta_5)$; thus, $S_j$ and $S_{-j-1}$ are isomorphic and therefore the isomorphism classes are represented by $S_j$ where $1 \leqslant j \leqslant \tfrac{q-3}{2}$.\end{proof}

\begin{propn} \label{ttex}
There exists a closed family $\bar{\mathscr{K}}_g$ of compact Riemann surfaces with a group of automorphisms isomorphic to the dihedral group of order $2q$ acting with signature $(0; 2,2,q,q).$ The number of equisymmetric strata of 
$\bar{\mathscr{K}}_g$ is at most \begin{displaymath}
 \left\{ \begin{array}{ll}
\tfrac{q+3}{4} & \textrm{if $q \equiv 1 \mbox{ mod } 4$}\\
\tfrac{q+1}{4}  & \textrm{if $q \equiv 3 \mbox{ mod } 4$}
  \end{array} \right.
\end{displaymath}and, independently of $q,$ one of them equals  $\mathscr{C}_g.$ 
\end{propn}

\begin{proof}

Let $\Delta_4$ be a Fuchsian group of signature $(0; 2,2,q,q)$  presented as in \eqref{cuatro} and consider the dihedral group of order $2q$ \begin{equation*}\label{denuevo} G\cong \mathbf{D}_q=\langle r, s : r^q =s^2=(sr)^2=1 \rangle.\end{equation*}The existence of the family $\bar{\mathscr{K}}_g$ follows after considering the ske $$ \Phi : \Delta_4 \to G \,\, \mbox{ given by } \,\, (x_1, x_2,x_3,x_4) \mapsto (s,s,r^{-1},r).$$

Let us now assume that $\theta : \Delta_4 \to G$ is a ske representing the action of $G$ on $S \in \bar{\mathscr{K}}_g$. If  $\theta(x_1) = \theta(x_2)$ then, after a conjugation and after sending $r$ to an appropriate power of it, we see that $\theta$ is equivalent to $\Phi.$ On the other hand, if  $\theta(x_1) \neq \theta(x_2)$ then, after considering a suitable automorphism of $G$, we see that $\theta$ is equivalent to the ske $$\theta_i:=(s,sr,r^{i}, r^{-i-1}) \,\, \mbox{ for some } i \in \mathbb{Z}_q-\{-1,0\}.$$

The braid transformation $\Phi_3$ (see \S\ref{vecinos}) shows that $\theta_i \cong \theta_{-i-1}.$ The rule $i \mapsto -i-1$ has order two, restricts to a bijection of $\mathbb{Z}_q-\{-1,0\}$ and has exactly  one fixed point; namely $i^*=\tfrac{q-1}{2}$.  Observe that if $\varphi_{u}$ is the automorphism of $G$ given by $r \mapsto r^u$ then $$\Phi = \Phi_2^2 \circ \varphi_{(i^*)^{-1}} (\theta_{i^*}).$$All the above ensures that $\theta$ is equivalent to either $$\Phi \, \, \mbox{ or }\,\, \theta_i \,\, \mbox{ for some } i \in \{1, \ldots, \tfrac{q-3}{2}\}.$$

Now, for each $i \in \{1, \ldots, \tfrac{q-3}{2}\}$ the transformation $\varphi_{i^{-1}} \circ \Phi_2^2$ provides an equivalence  $$\theta_i \cong \theta_{-i(2i+1)^{-1}}.$$ The rule $i \mapsto -i(2i+1)^{-1}$ has order two and  (up to the identification $i \sim -i-1$) restricts to a bijection of $\{1, \ldots, \tfrac{q-3}{2}\};$ it has a fixed point if and only if \begin{equation}\label{cuadratica}-i-1=-i(2i+1)^{-1} \,\, \iff \,\, 2i^2+2i+1=0\end{equation}and the quadratic equation above has solution in $\mathbb{Z}_q$ if and only if $q \equiv 1 \mbox{ mod }4.$ It follows that the number of pairwise non-equivalent skes $\theta$ is at most $$1+\tfrac{1}{2}( \tfrac{q-3}{2})=\tfrac{q+1}{4}   \,\, \mbox{ and }\,\, 2+\tfrac{1}{2} (\tfrac{q-3}{2}-1)=\tfrac{q+3}{4}$$if $q \equiv 3 \mbox{ mod 4}$ and $q \equiv 1 \mbox{ mod 4}$ respectively. Finally, with the notations of \eqref{g44}, define $$\hat{x}_1:= y_1y_2y_1, \,\, \hat{x}_2:=y_2, \,\, \hat{x}_3:= y_2y_1y_4y_1y_2 \,\, \mbox{ and }\,\, \hat{x}_4:= y_4 $$and notice that they generate a Fuchsian group isomorphic to $\Delta_4.$ The restriction of \eqref{antes} to it is given by \begin{equation} \label{opa}(\hat{x}_1,\hat{x}_2,\hat{x}_3,\hat{x}_4) \mapsto (T, T, R^{1-q}, R^{q-1}).\end{equation}If we write $s:=T$ and $r:=R^{q-1}$ we see that \eqref{opa} agrees with $\Phi.$ Hence, the action of $\Phi$ extends to \eqref{antes} and therefore the stratum defined by $\Phi$ agrees with $\mathscr{C}_g.$

\end{proof}
 
\begin{propn} If $\mathscr{K}_g$ stands for the interior of the closed family $ \bar{\mathscr{K}}_g$ then the full automorphism group of  $S \in \mathscr{K}_g$ is isomorphic to either $\mathbf{D}_q$ or $\mathbf{D}_{2q}.$ In addition \begin{displaymath}
 \bar{\mathscr{K}}_g-\mathscr{K}_g= \left\{ \begin{array}{ll}
\{X_8, X_4\} & \textrm{if $q \equiv 1 \mbox{ mod } 4$}\\
\,\,\,\,\{X_8\}   & \textrm{if $q \equiv 3 \mbox{ mod } 4.$}
\end{array} \right.
\end{displaymath}
\end{propn}

\begin{proof} We keep the  notations of the proof of Proposition \ref{ttex}. The first statement is clear since the full automorphism group  of $S$ is isomorphic to $\mathbf{D}_{2q}$ or $\mathbf{D}_{q}$ according to whether or not $\theta_i$ is equivalent to $\Phi.$ 

\s

Let ${\mathscr{K}}_g^i$ denote the equisymmetric stratum defined by $\theta_i.$

\begin{enumerate}
\item If $\theta_i$ is equivalent to $\Phi$ then Propositions \ref{ttex} and \ref{valp} imply that $\bar{\mathscr{K}}_g^i-\mathscr{K}_g^i=\{X_8\}.$ 

\item If $\theta_i$ is non-equivalent to $\Phi$ then, 
 by Proposition \ref{agro}, we see that:
 \begin{enumerate}
 \item if $q \equiv 3 \mbox{ mod }4$ then $\bar{\mathscr{K}}_g^i-\mathscr{K}_g^i$ is empty, and
 \item if $q \equiv 1 \mbox{ mod } 4$ then $\bar{\mathscr{K}}_g^i-\mathscr{K}_g^i$ is either empty or $\{X_4\}.$ 
\end{enumerate}
\end{enumerate}

Assume $q \equiv 1 \mbox{ mod }4.$ We recall that the full automorphism group of $X_4$ is isomorphic to \eqref{gratis} and the corresponding action is given by the ske \begin{equation} \label{src} \Delta_3 \to \mbox{Aut}(X_4) \,\, \mbox{ such that }\,\,(y_1, y_2, y_3) \mapsto (A^{-1}B, B^{-1}, A)\end{equation}where $\Delta_3$ is as in \eqref{jcc}.
The elements $$\hat{x}_1:= y_1y_2^2y_1^{-1}, \,\, \hat{x}_2:= y_1^2, \,\, \hat{x}_3:= y_1^{-1}y_3y_1 \,\, \mbox{ and }\,\, \hat{x}_4:=y_3$$generate a Fuchsian group isomorphic to $\Delta_4.$ The restriction of \eqref{src} to it is equivalent to  \begin{equation}\label{naran}(\hat{x}_1, \hat{x}_2,\hat{x}_3,\hat{x}_4) \to (B^2, B^2A, A^{e}, A^{-e-1}) \,\, \mbox{ where }e:=-\rho(\rho-1)^{-1}.\end{equation}By letting $r:=A$ and $s:=B^2$ we see that \eqref{naran} agrees with $\theta_{e}.$ As $e$ solves the  equation \eqref{cuadratica} we conclude that $\bar{\mathscr{K}}_g^i-{\mathscr{K}}_g^i$ equals $\{X_4\}$ if $i$ solves \eqref{cuadratica} and is empty otherwise. 
\end{proof}
 
\begin{propn}
If $\lambda=2$ then $S$ belongs to  $\bar{\mathscr{K}}_g$ or $S$ is isomorphic to one of the 
$\tfrac{q-3}{2}$ pairwise non-isomorphic surfaces  of Proposition \ref{lapiz2}.\end{propn}

\begin{proof}
Assume that the signature of the action of $G$ on $S$ is $(\gamma; k_1, \ldots, k_l).$ Observe that $\gamma=0.$ Indeed, otherwise  the Riemann-Hurwitz formula implies that $$l \leqslant 2(1-\tfrac{2}{q}) \,\, \mbox{ showing that } l=0 \mbox{ or } l=1.$$ In both cases we see that $\gamma=1$ and therefore $q=2$ and 4 respectively; a contradiction.  

As each $k_j  \geqslant 2,$ we see that $l \leqslant 6-\tfrac{4}{q}$ and therefore $l \in \{3,4,5\}.$
Let $v=\#\{k_j : k_j=2\}.$
\begin{enumerate}

\item If $l=5$ then $$\Sigma_{j=1}^5 \tfrac{1}{k_j}=2+\tfrac{2}{q} \leqslant \tfrac{v}{2}+\tfrac{5-v}{q}$$showing that $v=5.$ But, in this case $q=4;$ contradicting the fact that $q$ is prime.

\s

\item  If $l=4$ then $$\Sigma_{j=1}^4 \tfrac{1}{k_j}=1+\tfrac{2}{q} \leqslant \tfrac{v}{2}+\tfrac{4-v}{q}$$showing that $v\in \{2,3,4\}.$ Note that if $v=4$ then $q=2$ and if $v=3$ then $k_4 < 0.$ It follows that $v=2$ and therefore  $k_3=k_4=q.$

\s

\item  If $l=3$ then clearly $v \neq 3.$ If $v=2$ then $k_3 <0$ and if $v=1$ then $q < 4.$ Thus, \begin{equation}\label{mica}\tfrac{1}{k_1}+\tfrac{1}{k_2}+\tfrac{1}{k_3}=\tfrac{2}{q} \,\, \mbox{ where }\,\, k_j \in \{q, 2q\}.\end{equation}It is easy to verify that the unique solution of \eqref{mica} is $k_1=q, k_2=k_3=2q.$ 
\end{enumerate}

\s

Thereby, the signature of the action of $G$ on $S$ is either $(0; 2,2,q,q)$ or $(0; q,2q,2q).$

\s

We record here the simple fact that a group of order $2q$ is either cyclic or dihedral. 

\begin{enumerate}
\item If the signature is $(0; q,2q,2q)$ then $G$ is cyclic  and therefore, by Proposition \ref{lapiz2}, we obtain that $S \cong X_8 \in \bar{\mathscr{K}}_g$ or $S$ is isomorphic to one of the $\tfrac{q-3}{2}$ pairwise non-isomorphic surfaces with full automorphism group isomorphic to $C_q \times C_2.$

\s

\item If the signature is $(0; 2,2,q,q)$ then:
\begin{enumerate}
\item If $G$ is cyclic then, by Proposition \ref{lapiz}, we have that $S \in \bar{\mathscr{C}}_g \subset \bar{\mathscr{K}}_g.$ 

\item If $G$ is dihedral then, by Proposition \ref{ttex}, we have that $S \in \bar{\mathscr{K}}_g.$   
\end{enumerate}
\end{enumerate}This proves the proposition.
\end{proof}

\begin{remark} \label{casoq} Let $S \in \mathscr{M}_{q-1}^{q}$ and denote by $t$ the  automorphism of $S$ of order $q$.  According to \cite{scand}, the action of $G=\langle t \rangle$ on $S$ is equivalent to the action represented by one of the following skes:$$\theta_1=(t,t,t,t^{-3}) \,\, \,\,   \theta_2=(t,t,t^{-1},t^{-1})$$ $$   \theta_{3,i}=(t,t^{-1},t^{i}, t^{-i}) \,\, \,\,  \theta_{4,i}=(t,t,t^{i},  t^{q-2-i}) \,\, \,\,  \theta_{5,ij}=(t,t^{i},t^{j},  t^{q-1-i-j})$$where $2 \leqslant i \leqslant \tfrac{q-1}{2},$ $j \notin \{1, q-1, i, q-i\}$ and $q-1-i-j \notin \{1,i,j\}.$
\end{remark}

In terms of our terminology the results of \cite{scand} allow us to claim that the stratum determined by $\theta_1$ contains $X_3,$ the stratum determined by $\theta_2$ agrees with the family $\mathscr{C}_g,$ the strata determined by $\theta_{3i}$ agree with the equisymmetric strata of the family $\bar{\mathscr{K}}_g$ that are different from $\mathscr{C}_g,$ and the strata determined by $\theta_{4i}$ contain the surfaces $X_{2,k}.$ We also mention that the isolated strata of dimension one of $\mbox{Sing}(\mathscr{M}_{q-1})$ are the ones determined by the skes $\theta_{5ij}.$ 

\section{Proof of Theorem \ref{t2} and Proposition \ref{moduli}} \label{sprooft2}

Let $q \geqslant 5$ be a prime number and set $g=q-1.$ We write $\omega_l=\mbox{exp}(\tfrac{2\pi i}{l}).$

\subsection*{The family $\bar{\mathscr{C}}_g$} We recall that if $S \in \bar{\mathscr{C}}_g$ then 
the action of$$ G=\mathbf{D}_{2q}=\langle R,T : R^{2q}=T^2=(TR)^2=1 \rangle$$on $S$ has signature $(0; 2,2,2,q)$ and is represented by the ske  $\theta =( R^q, T, TR, R^{q-1})$. Let $H =\langle R^{q}\rangle \cong C_2.$ We consider the associated two-fold regular covering map $$\pi : S \to \Sigma:=S/H$$ and observe that $\Sigma$ has genus zero and $\pi$ ramifies over $2q$ values. If we  denote them by \begin{equation}\label{uv}u_1, \ldots, u_q \,\, \mbox{ and }\,\, v_1 \ldots, v_q\end{equation}then it is classically known that $S$ is isomorphic to the normalisation of  $$y^2= \Pi_{i=1}^{q}(x-u_i)(x-v_i)$$ 

Note that $ \Sigma \cong \mathbb{P}^1$ admits an  action of $K=G/H \cong \mathbf{D}_q$ in such a way that $\Sigma/K \cong S/G.$ It follows that   \eqref{uv} form two orbits of length $q$  under the action of the cyclic subgroup of order $q$ of $K$. Without loss of generality, we can assume that $$u_i=\omega_q^i \,\, \mbox{ and }\,\, v_i=\lambda \omega_q^i \,\, \mbox{ where }  \,\,1 \leqslant i \leqslant q$$and $\lambda$ is a nonzero complex number such that $t:=\lambda^q \neq 1.$ Hence, $S$ is isomorphic to the normalisation of the singular affine algebraic curve \begin{equation*}\label{hipe}\mathscr{X}_t :=\{(x,y) \in \mathbb{C}^2 : y^2=(x^q-1)(x^q-t) \} \end{equation*}for some $t \neq 0,1.$  It is a direct computation to verify that the transformations $$\mathfrak{r}(x,y)= (\omega_q x, -y) \,\, \mbox{ and }\,\, \mathfrak{t}(x,y) = (\sqrt[q]{t}\tfrac{1}{x}, \sqrt{t}\tfrac{y}{x^{q}})$$restrict to automorphisms of $\mathscr{X}_t$ and that  $\langle \mathfrak{r}, \mathfrak{t} \rangle \cong \mathbf{D}_{2q}.$ 

\s

Note that $\mathfrak{r}^q(x,y)=(x,-y)$ is the hyperelliptic involution and that $X_8\cong\mathscr{X}_{-1}.$

\subsection*{The surface $X_4$} Following the proof of Proposition  \ref{cargados}, the action of $$G=C_q \rtimes_4 C_4 =\langle A,B : A^q =B^4=1, BAB^{-1}=A^{\rho} \rangle$$ on $X_4$ has signature $(0; 4,4,q)$ and is represented by the ske $\theta=(A^{-1}B, B^{-1}, A).$ Let $Q =\langle a\rangle$ and observe that the associated $q$-fold regular covering map $$\pi : X_4 \to \Sigma:=X_4 / Q \cong \mathbb{P}^1$$ ramifies over four values marked with $q.$ As $\Sigma$ admits the action of $K=G/Q \cong C_4$ with signature $(0; 4,4)$ and $\Sigma/K \cong S/G,$ the branch values of $\pi$ form  one orbit under the action of $K$. Without loss of generality, we can  assume these branch values  to be $1,i,-1$ and $-i$ (where $i^2=-1$)  and that their corresponding rotation numbers  (modulo $q$) are  $1,\rho,\rho^2$ and $\rho^3$ respectively. Then, following  \cite{Gab} (see also \cite{H} and \cite{W}), $X_4$ is isomorphic to   the normalisation of \begin{equation}\label{naranja}y^q=(x-1)(x-i)^{\rho}(x+1)^{q-1}(x+i)^{q-\rho}.\end{equation}
Since $\rho$ has order 4 in $\mathbb{Z}_q,$ there exists $e \in \mathbb{Z}$ such that $\rho^2+1=eq.$ Set $$\mathfrak{a}(x,y) = (x, \omega_q y) \,\, \mbox{ and }\,\, \mathfrak{b}(x,y) = (ix,\tfrac{-(x+i)^{e-\rho}}{(x-i)^{e-1}(x+1)^{\rho -1}}y^{\rho})$$and notice that they restrict to automorphisms of \eqref{naranja}. If $\varphi$ is as in the statement of the theorem, then routine computations show that $$\varphi(-ix)\varphi(-x)^{\rho}\varphi(ix)^{\rho^2}\varphi(x)^{\rho^3}=y^{1-\rho^4}$$ and this implies that $\mathfrak{b}$ has order four. Now, it is direct to see that  $\langle \mathfrak{a}, \mathfrak{b} \rangle \cong C_q \rtimes_4 C_4.$

\begin{remark} \label{remax4}
The M\"{o}bius transformation $\iota :\mathbb{P}^1 \cong \Sigma \to \mathbb{P}^1$ given by $$\iota(z)=i\tfrac{z-1}{z+1} \,\,  \mbox{ satisfies } \,\, \iota(1,i,-1,-i) = (0,-1,\infty,1)$$and lifts to obtain an isomorphism between $X_4$ and the
 Riemann surface given by $$y^q=x(x+1)^{\rho}(x-1)^{q-\rho}$$This provides the explicit  model for $X_4$ defined over its field of moduli $\mathbb{Q}$. 
\end{remark}

\subsection*{The surface $X_3$} Similarly as before, the normalisation of $ y^3=x^q-1$ defines a Riemann  surface of genus $g$ and $(x,y) \mapsto (\omega_{q}x, \omega_{3}y)
$ restricts to an automorphism of it of order $3q.$  

\subsection*{The surfaces $X_{2,k}$} Following the proof of Proposition \ref{lapiz2}, the action of $$G=C_q \times C_2 = \langle a, b : a^q =b^2=[a,b]=1 \rangle$$on $X_{2,k}$ has signature $(0; q,2q,2q)$ and is determined by the ske 
$\theta_k= (a,a^kb, a^{-k-1}b)$ where $1 \leqslant k \leqslant \tfrac{q-3}{2}.$ If $Q =\langle a\rangle$ then the associated regular covering map $$\pi_k :  X_{2,k} \to \Sigma_k:=X_{2,k} / Q \cong \mathbb{P}^1$$  ramifies over four values marked with $q.$ As $\Sigma_k$ admits the action of $K=G/Q \cong C_2$ with signature $(0; 2,2)$ and $\Sigma/K \cong S/G,$   two branch values of $\pi_k$  form an orbit  and the remaining ones are fixed  under the action of $K$. Thus, we can assume that the branch values of $\pi_k$ are $1,-1, \infty$ and  $0,$  where the first two form an orbit. It follows that $X_{2,k}$ is isomorphic to\begin{equation}\label{naranjas}y^q=n^{n_k}(x-1)(x+1)\end{equation}for some $1 \leqslant n_k \leqslant q-1$ such that $n_k \neq q-2.$ It is easy to see that $$\mathfrak{a}(x,y) = (x, \omega_{q}y) \,\, \mbox{ and }\,\, \mathfrak{b}_k(x,y) = (-x, (-1)^{n_k}y)$$ are automorphisms of \eqref{naranjas} and that $\langle \mathfrak{a}, \mathfrak{b}_k \rangle \cong C_q \times C_2.$

\subsection*{The family $\bar{\mathscr{K}}_g$} Following the proof of Proposition \ref{ttex}, the action of \begin{equation} \label{dulces}G= \mathbf{D}_q= \langle r, s : r^q =s^2=(sr)^2=1 \rangle\end{equation}on $S \in \bar{\mathscr{K}}_g$ has signature $(0; 2,2,q,q)$ and is determined by the ske $ \theta_i =(s,sr,r^{i}, r^{-i-1})$ for some $1 \leqslant i \leqslant \tfrac{q-1}{2}.$  If $Q =\langle r\rangle$ then the associated regular covering map $$\pi: S \to \Sigma:=S / Q \cong \mathbb{P}^1$$  ramifies over four values marked with $q.$ As $\Sigma$ admits the action of $K=G/Q \cong C_2$ with signature $(0; 2,2)$ and $\Sigma/K \cong S/G,$ the branch values of $\pi$ form two orbits under the action of $K.$ We can assume these values to be $1,-1$ and $t,-t$ for some $t \neq 0, \pm 1$. In addition,  the rotation numbers are 1, $q-1$, 1 and $q-1$ respectively.  Thus, $S$ is isomorphic to $$\mathscr{Z}_t :=\{ (x,y) \in \mathbb{C}^2 :  y^q=(x-1)(x+1)^{q-1}(x-t)(x+t)^{q-1}\}$$It is straightforward to verify that the transformations $$\mathfrak{r}(x,y) = (x, \omega_q y) \,\, \mbox{ and }\,\, \mathfrak{s}(x,y) \mapsto (-x, (x^2-1)(x^2-t^2)y^{-1})$$restrict to  automorphisms of $\mathscr{Z}_t$ and $\langle  \mathfrak{r}, \mathfrak{s} \rangle\cong \mathbf{D}_q.$ 

\subsection*{Proof of Proposition \ref{moduli}} Assume that $S \in \mathscr{K}_q$ and let $G$ be as in \eqref{dulces}. The covering  $$\Sigma \to \Sigma/K \cong S/G \,\, \mbox{ can be chosen as }\, z \mapsto z^2,$$and therefore if $S \cong \mathscr{Z}_t$ then the branch values of $S \to S/G$ are $\infty$ and $0$ marked with 2,  and $1$ and $t^2$  marked with $q$.

Let $\sigma \in \mbox{Gal}(\mathbb{C}/\mathbb{Q})$ and assume  $S_t$ and $(S_t)^{\sigma}=S_{\sigma(t)}$ to be isomorphic. The facts  that
\begin{enumerate}
\item $G \cong \mbox{Aut}(S)$  provided that $S \in \mathscr{K}_g-\mathscr{C}_g,$ and
\item $G$ is the unique group of automorphisms of $S$ isomorphic to $\mathbf{D}_q$ provided that $S \in \mathscr{C}_g$
\end{enumerate}imply that there is a M\"{o}bius transformation $\varphi: \mathbb{P}^1 \to \mathbb{P}^1$ such that  $$\varphi(\{\infty, 0 \})=\{\infty, 0 \}  \,\,\mbox{ and} \,\,\, \varphi(\{1,t^2\})=\{1,\sigma(t)^2\}.$$
We have two possible cases:

\begin{enumerate}
\item If $\varphi(0)=0$ and $\varphi(\infty)=\infty$ then either $\varphi(z)=z$ or $ \varphi(z)=\sigma(t)^2z.$ 

\s

\item If $\varphi(0)=\infty$ and $\varphi(\infty)=0$ then either $\varphi(z)=\tfrac{1}{z}$ or $\varphi(z)=\tfrac{\sigma(t)^2}{z}.$
\end{enumerate}

\s

Observe that the latter case in (1) and the former case in (2) imply that $t\sigma(t)=\pm 1;$ a contradiction. It follows that $\sigma(t)^2=t^2$ showing that $\sigma(t)=t.$ 
Conversely, if $\sigma(t)=t$ then it is clear that $S=(S_t)^{\sigma}.$ Hence, the field of moduli  of $S \cong \mathscr{Z}_t$ is  $$\mbox{fix} \{\sigma \in  \mbox{Gal}(\mathbb{C}/\mathbb{Q}) : \sigma(t)=t\}=\mathbb{Q}(t)$$as desired.

\section{Some algebraic lemmata} \label{lemmatas} In this section we collect some facts related to the representations of the groups appearing in Theorem \ref{tm}; these results will be needed to prove Theorems \ref{t3} and \ref{t4}. Set $\omega_l=\exp(\tfrac{2 \pi i}{l}).$

\subsection*{Rational and complex irreducible representations}
\begin{lemm} \label{anto} Let $q \geqslant 5$ be a prime number. The group  \begin{equation*} \label{grupoxocho} G_8= \langle x,y,z: x^{2q}=y^2=z^2=1, [x,y]=[z,y]=1, zxz=x^{-1}y \rangle\end{equation*} has four complex irreducible representations of degree one  given by\begin{displaymath}
\chi_0^{1+}: \left\{ \begin{array}{ll}
x \mapsto 1\\
y \mapsto 1 \\
z \mapsto 1
  \end{array} \right. \,\,\, \chi_0^{2+}: \left\{ \begin{array}{ll}
x \mapsto 1\\
y \mapsto 1 \\
z \mapsto -1
  \end{array} \right. \,\,\, \chi_q^{1+}: \left\{ \begin{array}{ll}
x \mapsto -1\\
y \mapsto 1 \\
z \mapsto 1
  \end{array} \right.\,\,\, \chi_q^{2+}: \left\{ \begin{array}{ll}
x \mapsto -1\\
y \mapsto 1 \\
z \mapsto -1
  \end{array} \right.
\end{displaymath}and $2q-1$ complex irreducible representations of degree two given by
\s

\begin{center}
\begin{tabular}{|c|c|c|c|}
\hline
$\operatorname{Representation}$&$x$ & $y$& $z$\\ \hline

$\chi^{+}_j, \, 1 \leqslant j \leqslant q-1$ &$\operatorname{diag}(\omega_{2q}^j, -\omega_{2q}^{q-j})$ & $I_2$ & $J_2$\\[4 pt]

$\chi^{1-}_j, \, 0 \leqslant j \leqslant \frac{q-1}{2}$ &$\operatorname{diag}(\omega_{2q}^j, \omega_{2q}^{q-j})$ & $-I_2$ & $J_2$\\[4 pt]

$\chi^{2-}_j, \, 1 \leqslant j \leqslant \frac{q-1}{2}$ &$\operatorname{diag}(-\omega_{2q}^j, -\omega_{2q}^{q-j}) $ & $-I_2$ & $J_2$\\[4 pt]
\hline
\end{tabular}
\end{center} 
\s
where $I_2$ stands for the $2 \times 2$ identity matrix and $J_2=\left( \begin{smallmatrix}
0 & 1  \\
1 &  0 \\
\end{smallmatrix} \right).$ 
\end{lemm}

\begin{proof} The proof  is  an application of the method of Wigner and Mackey  to built the complex irreducible representations of certain semidirect products. See, for example, \cite[\S 8.2]{serre}.
\end{proof}

\begin{lemm} \label{anto2} Let $q \geqslant 5$ be a prime number such that $q \equiv 1  \mbox{ mod }4.$ Let $\rho$ be a primitive fourth root of unity in $\mathbb{Z}_q$ and choose a maximal subset $\mathcal{P} \subset \mathbb{Z}_q^*$ of representatives of the  relation $k \sim k \rho  \sim -k\sim - k \rho$ over $\mathbb{Z}_q^*$. The group  $$G_4=\langle A, B: A^{q}=B^4=1, BAB^{-1}=A^{\rho} \rangle$$has four complex irreducible representations of degree one given by\begin{displaymath}
\chi_{1}: \left\{ \begin{array}{ll}
A \mapsto 1\\
B \mapsto 1
  \end{array} \right. \,\,\, \chi_{i}: \left\{ \begin{array}{ll}
A \mapsto 1\\
B \mapsto i
  \end{array} \right. \,\,\, \chi_{-1}: \left\{ \begin{array}{ll}
A \mapsto 1\\
B \mapsto -1
  \end{array} \right.\,\,\, \chi_{-i}: \left\{ \begin{array}{ll}
A \mapsto 1\\
B \mapsto -i
  \end{array} \right.
\end{displaymath}and $\tfrac{q-1}{4}$ complex irreducible representations of degree four, given by $$  \phi_j : A \mapsto \operatorname{diag}(\omega_q^{j}, \omega_q^{\rho j}, \omega_q^{-j}, \omega_q^{-\rho j}), \,\,  B \mapsto \left( \begin{smallmatrix}
0 & 0  & 0 & 1\\
1 &  0 & 0 & 0 \\
0 &  1 & 0 & 0 \\
0 &  0 & 1 & 0 \\
\end{smallmatrix} \right)$$for  $
j \in \mathcal{P}$. In addition, the rational irreducible representations of $G_4$ are $$\chi_1, \,\, \chi_{-1}, \,\, \chi_i \oplus \chi_{-i} \,\, \mbox{ and } \,\, \varrho = \oplus_{j \in \mathcal{P}} \phi_j$$
\end{lemm}

\begin{proof} The construction of the representations follows from \cite[\S 8.2]{serre} and we only need to prove the last statement. As the $G_4$  has four conjugacy classes of cyclic subgroups, it has  four pairwise non-equivalent rational irreducible representations: three of them are $\chi_1, \chi_{-1}$ and $\chi_i \oplus \chi_{-i}$. It follows that $\phi_j$ are Galois conjugate and added up together produce the remaining rational irreducible representation.
\end{proof}

\subsection*{Symmetric square's character formula} Let $H$ be a finite group and let $\rho: H \to \mbox{GL}(V)$ be a complex representation of $H.$ Consider the associated  representation of $H$ on the symmetric square vector space of $V$ $$\mbox{Sym}^2(\rho) : H \to \mbox{GL}(\mbox{Sym}^2(V)).$$ According to \cite[Proposition 2.3]{serre}, the character $\chi_{\rho}^{\tiny \mbox{sym}}$ of $\mbox{Sym}^2(\rho)$ is given by \begin{equation}\label{eq:sim67}
 \chi_{\rho}^{\sy}(h)=\tfrac{1}{2}[\chi_{\rho}(h)^2+\chi_{\rho}(h^2)] \,\, \mbox{ for }\,\, h \in H,\end{equation}where $\chi_{\rho}$ denotes the character of $\rho.$ 

\begin{lemm}\label{L:potatoeN} Let $\chi$ be a character of $H$ and let $\bar{\chi}$ denote its complex-conjugate. Then
\begin{equation*} \label{mora}\Sigma_{h \in H}\chi^{\sy}=\tfrac{1}{2}[\Sigma_{h \in H}(\chi + \bar{\chi})^{\sy}(h)-|H|\langle \chi|\chi\rangle_H]\end{equation*}
\end{lemm}

\begin{proof} According to \cite[Exercise 2.1]{serre}, for any pair of characters $\chi_1$ and $\chi_2$ we have\begin{equation*} (\chi_1 + \chi_2)^{\sy}=\chi_1^{\sy}+\chi_2^{\sy}+\chi_1\chi_2.\end{equation*}
If we write $\chi_1=\chi$ and $\chi_2=\bar{\chi}$  then the previous equality implies\begin{equation*}
\Sigma_{h \in H}\chi^{\sy}(h)=\Sigma_{h\in H}[(\chi + \bar{\chi})^{\sy}(h)-(\bar{\chi})^{\sy}(h)-\chi(h)\bar{\chi}(h)].
\end{equation*}Since $$\Sigma_{h\in H} (\bar{\chi})^{\sy}(h)=\Sigma_{h\in H} \chi^{\sy}(h^{-1})=\Sigma_{h\in H} \chi^{\sy}(h),$$the conclusion follows after noticing that $\Sigma_{h\in H}\chi(h)\bar{\chi}(h)=|H| \langle \chi | \chi\rangle_H$.
\end{proof}

\subsection*{The analytic representation} Let $S$ be a compact Riemann surface of genus $g$ and let $H$ be a group of automorphisms of $S$. The action of $H$  induces a complex representation $$\rho_a : H \to \mbox{GL}(\mathscr{H}^{1,0}(S, \mathbb{C})) \cong \mbox{GL}(g, \mathbb{C}),$$called the analytic representation of $H.$ Let $\text{Irr}(H)$ denote the set of complex irreducible representations of $H$, up to equivalence. If we write
$$\rho_a \cong \oplus_{\rho \in\text{Irr}(H)}\mu_{\rho} \rho \,\, \mbox{ where }\,\, \mu_{\rho} \in \mathbb{N} \cup \{0\}
$$then $\mu_{\rho}$ can be computed using the classically known Chevalley-Weil formula; see \cite{chevw}.

\begin{lemm} \label{chew} Assume that the action of $H$ on $S$ has signature $\sigma=(0; k_1, \ldots, k_s)$ and is represented by the ske $\theta : \Delta \to H$ where $\Delta$ is a Fuchsian group of signature $\sigma$ canonically presented as in \eqref{prese}. Then $\mu_{\rho}=0$ if $\rho$ is the trivial representation; otherwise   \begin{equation}\label{eq:chevw}
\mu_{\rho}=-d_{\rho}+\Sigma_{l=1}^s\Sigma_{j=1}^{k_l} N_{l,j}^{\rho}  ( 1 - \tfrac{j}{k_l})
\end{equation}where $d_{\rho}$ is the degree of $\rho$ and $N_{l,j}^{\rho}$ is the number of eigenvalues of $\rho(\theta(x_l))$ that equal $\omega_{k_l}^j$. \end{lemm}

\subsection*{Computation of the dimension $N_{S,G}$} 

Let $S$ be a compact Riemann surface of genus $g.$ As mentioned in  \S \ref{ff2}, following  \cite{ST} and \cite[Lemma 3.8]{paola} together with  the  formula \eqref{eq:sim67}, the dimension of  (the component which contains $JS$ of) the submanifold $\mathscr{S}_S$ of $\mathscr{H}_g$ of matrices representing ppavs admitting an action  equivalent to the one of $\mbox{Aut}(S)$   is 
\begin{equation} \label{Nsinrac}N_{S,G}=\langle  \chi_{\rho_{a}}^{\sy} | 1 \rangle_G=\tfrac{1}{|G|}\Sigma_{g \in G} \chi_{\rho_{a}}^{\sy}(g)=\tfrac{1}{2|G|}\Sigma_{g \in G} [\chi_{\rho_{a}}(g)^2+\chi_{\rho_a}(g^2)]\end{equation}where $\rho_a$ is the analytic representation of $G=\mbox{Aut}(S).$

\s

As a direct consequence of Lemma \ref{L:potatoeN}, the previous equality can be rewritten as follows.

\begin{lemm} \label{ffinal} 
If $\chi_{\rho_a}$ denotes the character of the analytic  representation of $G=\operatorname{Aut}(S)$ then $$N_{S,G}=\tfrac{1}{2|G|}[\Sigma_{g \in G}(\chi_{\rho_a} + \bar{\chi}_{\rho_a})^{\sy}(g)-|G|\langle \chi_{\rho_a}|\chi_{\rho_a}\rangle_G].$$
\end{lemm}

\begin{remark} \mbox{}
\begin{enumerate}

\item  The computation of $N_{S,G}$ depends both on the group $G$ and on its action on the Riemann surface $S.$ It then makes sense to compute $N_{S,H}$ for a subgroup $H$ of $G=\mbox{Aut}(S).$ When considering $G=\mbox{Aut}(S)$ we write $N_S$ instead of $N_{S,G}.$
 
\s
\item  Observe that if $H_1 \leqslant H_2$ are two groups of automorphisms of $S$ then  $N_{S,H_2} \leqslant N_{S, H_1}.$ In particular, if $N_{S,H}=0$ for some group of automorphisms $H$ of $S$ then $N_S=0.$
\s
\item According to \eqref{eq:sim67}, the summand  $(\chi_{\rho_a} + \bar{\chi}_{\rho_a})^{\sy}(g)$ in Lemma \ref{ffinal} equals\begin{equation}\label{eq:oplus}(
\chi_{\rho_a} + \bar{\chi}_{\rho_a})^{\sy}(g)=\tfrac{1}{2}[ (\chi_{\rho_a} + \bar{\chi}_{\rho_a})^2(g)+(\chi_{\rho_a} + \bar{\chi}_{\rho_a})(g^2)]
\end{equation}and is, indeed, the sum of two rational numbers.
\end{enumerate}
\end{remark}

\section{Proof of Theorems \ref{t3} and \ref{t4}} \label{pt3t4}

\subsection*{The surface $X_8$} With the notations of Lemma \ref{anto},  the representations $$\chi^{1-}_j \,\, \mbox{ and }\,\, \chi^{2-}_j \,\, \mbox{ for }\,\, j=1,\dots , \tfrac{q-1}{2}$$of $G_8$ are pairwise Galois conjugate, showing that their direct sum yield a rational irreducible representation of degree $2(q-1)$. We denote by $\varrho$ this last representation, namely$$\varrho \cong  ( \oplus_{j=1}^{\frac{q-1}{2}} \chi^{1-}_j ) \oplus (\oplus_{j=1}^{\frac{q-1}{2}} \chi^{2-}_j).$$

 As explained in \S\ref{jjo}, the group algebra decomposition of $JX_8$ with respect to  $G_8 \cong \mbox{Aut}(X_8)$ has the form $$JX_8 \sim B^n \times P$$where $B$ is the abelian subvariety of $JX_8$ associated to $\varrho$ and $P$ is the product of the factors associated to the remaining rational irreducible representations of $G_8.$ Following \cite[Proposition (10.8)]{isaacs}, the  Schur index of $\chi^{1-}_1$ is one  and then $n=2.$

\s

We recall  that the action of $G_8$ on $X_8$ is represented by the ske $\theta= (z,zx,x^{-1}).$  The dimension of the fixed subspaces of $\chi^{1-}_1$ under the action of the subgroups $\langle z \rangle, \langle zx \rangle$ and $\langle x^{-1} \rangle$ is $1,0$ and $0$ respectively; this is clear by noticing that  $\chi^{1-}_1(z)=J_2$ and that $$\chi^{1-}_1(zx)=\left(\begin{smallmatrix}0&  \omega_{2q}^{q-1} \\ \omega_{2q}  &0\\ \end{smallmatrix}\right) \,\, \mbox{ and } \,\, \chi^{1-}_1(x^{-1})= \operatorname{diag}(\omega_{2q}^{-1}, \omega_{2q}^{1-q})$$ do not have $1$ as an eigenvalue. In addition, it is easy to see that the character field of $\chi^{1-}_1$ has degree $q-1$ over the rationals. We then apply the equation \eqref{dimensiones1} to conclude that$$\dim B=(q-1)\left[-2+\tfrac{1}{2}\left( (2-1)+(2-0)+(2-0)\right)\right]=\tfrac{q-1}{2}.$$

Since the dimension of $JX_8$ is $q-1$, it follows that $P=0$ and therefore \begin{equation} \label{dia1}JX_8 \sim B^2.\end{equation}

Finally, we consider the subgroup $\langle z \rangle$ of $G_8$ and write $Y_8=X_8 / \langle z \rangle.$ The induced isogeny \eqref{decoind1} applied to \eqref{dia1} implies that $$JY_8=J(X_8/\langle z \rangle) \sim B \,\, \mbox{ and therefore }\,\, JX_8 \sim JY_8^2$$as claimed in Theorem \ref{t3}.

\s

We now proceed to prove that $N_{X_8}=0.$ Let $H = \langle x \rangle \cong C_{2q}$ and consider the maps$$\rho_k : H \to \mathbb{C} \,\, \mbox{ given by } \,\, x \mapsto \omega_{2q}^k \,\, \mbox{ for  } \,\, k=0, \ldots, 2q-1.$$

We claim that the analytic representation $\rho_a$  of $H$ decomposes as the direct sum$$\rho_a \cong \oplus_{j=q+1}^{2q-1} \rho_j.$$

To prove that, observe that $\{\rho_0,\dots,\rho_{2q-1}\}$ is a full set of pairwise non-equivalent complex irreducible representations of $H.$ Besides, as noticed in the proof of Proposition \ref{lapiz2}, the induced action of $H$ on $X_8$ has signature $(0; q, 2q,2q)$ and  is represented by the ske $(x^2,x^{-1}, x^{-1})$ (see also the algorithm in \cite{poly2} based on \cite{singerman2}). If we write$$\rho_a  \cong  \oplus_{k=0}^{2q-1} \mu_k \rho_k \,\, \mbox{ for some }\,\, \mu_k \in \mathbb{N} \cup \{0\}$$then, according to Lemma \ref{chew}, we have $\mu_0=0$ and  $$\mu_k=-1+ 2\Sigma_{j=1}^{2q-1}N_{1,j}^k ( 1 - \tfrac{j}{2q}) + \Sigma_{j=1}^{q-1}N_{2,j}^k( 1 - \tfrac{j}{q}),$$for each  $k\in \{1,\dots , 2q-1\},$ where $N_{1,j}^k=1$ if and only if $j=2q-k$  and 
$$N_{2,j}^k=\begin{cases} 1  \text{ if and only if } j=k  \text{ for } k\leqslant q \\ 1  \text{ if and only if } j=k-q  \text{ for } k\geqslant q+1.\\ \end{cases}$$In this way, we obtain if  $k\leqslant q$ then $\mu_k=0$ and if $k\geqslant q+1$ then $\mu_k=1.$ The claim follows.

\s

We then  can construct the following table
\s
\s

\begin{center}
\begin{tabular}{|c|c|c|c|}
\hline
 $h$ & order &$(\chi_{\rho_a} + \bar{\chi}_{\rho_a})^2(h)$ & $(\chi_{\rho_a} + \bar{\chi}_{\rho_a})(h^2)$ \\ \hline
1 & 1 & $(2(q-1))^2$ & $2(q-1)$\\ \hline
$x^q$ & $2$ &  $0$ & $2(q-1)$\\ \hline
$x^{2k}, \,\, 1\leqslant k \leqslant q-1$& $q$  & $(-2)^2$ & $-2$ \\ 
 \hline
$x^{2k-1}, \, 1\leqslant k \leqslant q, \, k \neq \tfrac{q+1}{2}$& $2q$ & $0$ & $-2$ \\ 

\hline
\end{tabular}
\end{center}
\s
\s
to obtain, by \eqref{eq:oplus}, that $\Sigma_{h\in H}(\chi_{\rho_a} + \bar{\chi}_{\rho_a})^{\sy}(h)$ equals$$\tfrac{1}{2}\left( 4(q-1)^2+2(q-1)+2(q-1)+2(q-1)-2(q-1)\right)=2(q^2-q).$$Now, the fact that $|H| \langle \chi_{\rho_a} | \chi_{\rho_a}\rangle_H=2q(q-1),$ together with Lemma \ref{L:potatoeN}   imply that$$\Sigma_{h\in H}\chi_{\rho_a}^{sym}(h)=\tfrac{1}{2}[2(q^2-q)-2q(q-1)]=0$$and therefore we obtain $N_{X_8, H}=0.$ Hence $N_{X_8}=0$ as claimed in Theorem \ref{t4}.

\subsection*{The surface $X_4$} With the notations of Lemma \ref{anto2}, the group algebra decomposition of $JX_4$ with respect to  $G_4 \cong \mbox{Aut}(X_4)$ is $$JX_4 \sim D_{1} \times D_{i}\times D_{-1}  \times D^4$$where the factor $D_k$ is associated to the representation $\chi_k$ and $D$ is associated to $\varrho.$ Observe that the character field of each $\phi_j$ has degree $\tfrac{q-1}{4}$ over the rationals.

   We recall  that the action of  $G_4$ on $X_4$ is represented by the ske $ \Theta = (A^{-1}B, B^{-1}, A)$. The dimension of the fixed subspace of $\phi_j$ under the action of $\langle A^{-1}B \rangle, \langle B^{-1}\rangle$ and $\langle A \rangle$ is $1, 1$ and $0$ respectively. Consequently,  the equation  \eqref{dimensiones1} implies $$\dim D=\tfrac{q-1}{4}\left[-4+\tfrac{1}{2}\left((4-1)+(4-1)+(4-0)\right)\right]=\tfrac{q-1}{4}.$$The previous equality shows, in addition, that $$D_{1} = D_{i}= D_{-1}=0 \, \, \mbox{ and therefore } \,\, JX_4\sim D^4.$$

Finally, we consider the subgroup $\langle B \rangle$ of $G_4$ and write $Y_4=X_4 / \langle B \rangle.$ The induced isogeny \eqref{decoind1} applied to the previous isogeny implies that $$JY_4=J(X_4/\langle B \rangle) \sim D \,\, \mbox{ and therefore }\,\, JX_4 \sim JY_4^4$$as claimed in Theorem \ref{t3}.

\s

We now proceed to prove that $N_{X_4}=\tfrac{q-1}{4}.$ A routine application of Lemma \ref{chew} shows that the analytic representation $\rho_a$ of $G_4$ is $$\rho_a \cong \oplus_{j \in \mathcal{P}} \phi_j \,\, \mbox{ and therefore }\,\, |G_4|\langle \chi_{\rho_a}|\chi_{\rho_a}\rangle_{G_4}=q(q-1)$$where $\mathcal{P}$ is as in Lemma \ref{anto2}. It is not difficult to see that\begin{displaymath}
(\chi_{\rho_a} + \bar{\chi}_{\rho_a})(g) = \left\{ \begin{array}{ll}
 2(q-1) & \textrm{if $g=1$}\\
\hspace{0.6 cm} 0 & \textrm{if $|g|=2,4$}\\
\hspace{0.3 cm} -2 & \textrm{if $|g|=q$}
  \end{array} \right.
\end{displaymath}and therefore \eqref{eq:oplus} allows us to write that $(\chi_{\rho_a} + \bar{\chi}_{\rho_a})^{\sy}(1)=(q-1)(2q-1)$ and that\begin{displaymath}
\Sigma_{g\in G_4,\; |g|=\epsilon}(\chi_{\rho_a} + \bar{\chi}_{\rho_a})^{\sy}(g) = \left\{ \begin{array}{ll}
q(q-1) & \textrm{if $\epsilon=2$}\\
\hspace{0.6 cm} 0 & \textrm{if $\epsilon=4$}\\
\hspace{0.3 cm} q-1 & \textrm{if $\epsilon=q.$}
  \end{array} \right.
\end{displaymath}

Thus, $\Sigma_{g\in G_4}(\chi_{\rho_a} + \bar{\chi}_{\rho_a})^{\sy}(g)=3q(q-1).$ Now, it follows from Lemma \ref{L:potatoeN}  that$$\Sigma_{g\in G_4}\chi_{\rho_a}^{\sy}(g)=\tfrac{1}{2}\left(3q(q-1)-q(q-1)\right)=q(q-1),$$
and consequently Lemma \ref{ffinal} says that $N_{X_4}=\tfrac{q-1}{4},$ as desired.

\subsection*{The surface $X_3$} The complex irreducible representations of $G_3=\langle x : x^{3q}=1 \rangle$ are  $$\chi_k: G_3 \to \mathbb{C}, \,\, x \mapsto \omega_{3q}^k \text{ for } k \in \{0, \dots, 3q-1\}.$$

We denote by $\rho_{r}$ the complexification of the rational representation   corresponding  to the action of $G_3$ on $X_3.$ We claim that \begin{equation*}\label{eq:rho_rac_X3}\rho_{r}\cong \oplus_{\sigma \in \mathscr{G}} \chi_1^\sigma \,\, \mbox{ where }\,\, \mathscr{G}=\Gal(\mathbb{Q}(\omega_{3q})/\mathbb{Q}) \end{equation*}

In fact, according to \cite[Theorem 5.10]{yoibero}, the multiplicity of $\chi_1$ in the decomposition of $\rho_r$ as a sum of irreducible representations equals one. In addition, since $\rho_r$ is indeed defined over the rationals we can  deduce that all the orbit of $\chi_1$ under $\mathscr{G}$ appears in the decomposition of $\rho_{r}$. The claim follows after noticing that the aforementioned orbit has length $2(q-1)$ and this number agrees with the degree of $\rho_r.$ 

\s

Since $\rho_r \cong \rho_a \oplus \bar{\rho}_a,$ the previous claim says that 
$\rho_a$ decomposes into $q-1$ pairwise non-equivalent  complex irreducible representations of degree one of $G_3$ and thereby\begin{equation} \label{analiticas}\langle \chi_{\rho_a} | \chi_{\rho_a}\rangle_{G_3}=q-1.\end{equation}

In order to determine the character $\chi_{\rho_r}$ of $\rho_r$ it is convenient to decompose  $\rho_r$ in the following different but equivalent way. Let 
$$\Lambda(q)= \{ t \in \{1,\dots, 3q-1\} : \mbox{gcd}(t, 3q)=1 \}$$and consider the subset  $\Lambda'(q)$ of $\Lambda(q)$ of cardinality $q-1$  obtained by removing the additive inverses modulo $3q$ (that is, $k\in \Lambda'(q)$ then $-k \text{ mod } 3q \notin \Lambda'$). It is not difficult to see that $$\rho_{r}\cong \oplus_{k\in \Lambda'(q)} (\chi_k \oplus \bar{\chi}_k) \,\, \mbox{ and therefore }\,\,\chi_{\rho_r}=\Sigma_{k\in \Lambda'(q)} (\chi_k+\bar{\chi}_k).$$

\begin{enumerate}
\item Clearly $\chi_{\rho_r}(1)=2(q-1);$ the degree of $\rho_r.$
\s

\item For the elements of order $3$ (that is, $x^{q}$ and $x^{2q}$) we have that
$$ (\chi_k+\bar{\chi}_k)(x^q)=\omega_{3q}^{kq}+\omega_{3q}^{-kq}=-1 \text{ for each } k\in \Lambda' \,\, \implies \chi_{\rho_r}(x^q)=-(q-1). $$Analogously, one sees that $\chi_{\rho_r}(x^{2q})=-(q-1)$.

\s
\item For the elements of order $q$ (that is, $x^{3j}$ with $j\in \{1,\dots,q-1\}$) we have that$$ (\chi_k+\bar{\chi}_k)(x^{3j})=\omega_{3q}^{3kj} +\omega_{3q}^{-3kj}  = \omega_{q}^{kj} +\omega_{q}^{-kj}$$ for every $k\in \Lambda'$. 
Then$$\chi_{\rho_r}(x^{3j})=\Sigma_{k\in \Lambda'(q)} (\chi_k+\bar{\chi}_k)(x^{3j})=\Sigma_{k\in \Lambda'(q)}(\omega_{q}^{kj} +\omega_{q}^{-kj}),$$showing that$$\chi_{\rho_r}(x^{3j})=-2 \, \text{ for all } \, j \in \{1,\dots, q-1\}.$$

\item For the elements of order $3q$ (that is, $x^t$ with $t\in \Lambda(q)$) we have that $$\chi_{\rho_r}(x^t)=\Sigma_{k\in \Lambda'(q)} (\chi_k+\bar{\chi}_k)(x^t)=\Sigma_{k\in \Lambda'(q)}(\omega_{3q}^{kt}+\omega_{3q}^{-kt})$$ and this corresponds to the sum of all primitive $3q$-th  roots of unity. 
It is a known fact that this sum corresponds to the M\"{o}bius function $\mu(3q)$, which is $1$; thus$$\chi_{\rho_r}(x^t)=1 \,\, \mbox{ for every }\,\,t\in \Lambda(q).$$
\end{enumerate}

We summarise all the above in the third column of the following table; the fourth column follows from all the above and \eqref{eq:oplus}. 

\s
\s
\begin{center}
\begin{tabular}{|c|c|c|c|}
\hline
order & $g$ & $\chi_{\rho_r}(g)= (\chi_{\rho_a} + \bar{\chi}_{\rho_a})(g)$ & $2(\chi_{\rho_a} + \bar{\chi}_{\rho_a})^{\sy}(g)$\\ \hline
1 & 1 & $2(q-1)$ & $2(q-1)+(2(q-1))^2$\\ \hline
3 & $x^q, x^{2q}$ & $-(q-1)$ &  $-(q-1)+(q-1)^2$ \\ \hline
$q$ & $x^{3j}, \, j=1,..q-1$ & $-2$& $-2+4=2$\\
 \hline
 $3q$ & $x^t, \,t \in \Lambda(q)$ & $1$& $1+1=2$\\
  \hline
\end{tabular}
\end{center}
\s
\s

If follows that$$\Sigma_{g\in G_3}(\chi_{\rho_a} + \bar{\chi}_{\rho_a})^{\sy}(g)=3q(q-1)$$and the desired result $N_{X_3}=0$ follows  from Lemma \ref{L:potatoeN} together with the equality \eqref{analiticas}.

\subsection*{The surfaces $X_{2,k}$} Let $1 \leqslant k \leqslant \tfrac{q-3}{2}.$  Consider the complex irreducible representations of  $$\mbox{Aut}(X_{2,k}) \cong G_2 = \langle a, b: a^q=b^2=[a, b]=1 \rangle$$
given by $\chi_1: x \mapsto \omega_{2q}$ and $\chi_2: x \mapsto \omega_{q}$ where  $x:=ab.$ If $\mathscr{G}_1$ and $\mathscr{G}_2$ are the Galois groups of  the extensions of $\mathbb{Q}$ by $\mathbb{Q}(\omega_{2q})$ and $\mathbb{Q}(\omega_{q})$  respectively, then 
$$\oplus_{\sigma \in \mathscr{G}_1} \chi_1^{\sigma} \,\, \mbox{ and } \,\,  \oplus_{\sigma \in \mathscr{G}_2} \chi_2^{\sigma}$$are rational irreducible representations of $G_2$ of degree $q-1.$  

\s

Let $\rho_{r}$ denote the complexification of the rational representation   corresponding  to the action of $G_2$ on $X_{2,k}.$ By arguing as done in the case of $X_3,$ one obtains that 
\begin{equation*}\label{eq:rho_rac_X2j}
\rho_{r}\cong (\oplus_{\sigma \in \mathscr{G}_1} \chi_1^{\sigma}) \oplus (\oplus_{\sigma \in \mathscr{G}_2} \chi_2^{\sigma}),
\end{equation*}and if $\rho_a$ is the analytic representation of the involved action then \begin{equation}\label{copiapo}\langle \chi_{\rho_a} | \chi_{\rho_a}\rangle_{G_2}=q-1,\end{equation}where $\chi_{\rho_a}$ is the character of $\rho_a$. Moreover, the character $\chi_{\rho_r}$ of $\rho_{r}$ is given by\begin{displaymath}
\chi_{\rho_r} (g) = \left\{ \begin{array}{ll}
 2(q-1) & \textrm{if $g=1$}\\
\hspace{0.6 cm} 0 & \textrm{if $|g|=2$}\\
\hspace{0.3 cm} -2 & \textrm{if $|g|=q$}\\
\hspace{0.6 cm} 0& \textrm{if $|g|=2q$}
  \end{array} \right.
\end{displaymath}and therefore we can construct the following table.

\s
\s
\begin{center}
\begin{tabular}{|c |c| c| c|}
\hline
Order & $g$ & $\chi_{\rho_r}(g)= (\chi_{\rho_a} + \bar{\chi}_{\rho_a})(g)$ & $2(\chi_{\rho_a} + \bar{\chi}_{\rho_a})^{\sy}(g)$\\ \hline 
$1$ & $1$ & $2(q-1)$ & $2(q-1)+(2(q-1))^2$\\ \hline
$2$ & $b$ & $0$ &  $2(q-1)+0^2$ \\ \hline
$q$ & $a^j, \, j=1, \ldots, q-1$ & $-2$& $-2+(-2)^2=2$\\
 \hline
 $2q$ & $a^jb, \, \mbox{gcd}(j,2q)=1$ & 0& $-2+0^2$\\
  \hline
\end{tabular}
\end{center}
\s
\s

The desired conclusion $N_{X_{2,k}}=0$ follows from the previous table,  the equation \eqref{copiapo} and Lemma \ref{ffinal}.

\subsection*{The family ${\mathscr{C}}_g$} The complex representation $\psi : \mathbf{D}_{2q} \to \mbox{GL}(2, \mathbb{C})$ (with the dihedral group presented as in \eqref{frio})  given by $$\psi(R)=\mbox{diag}(\omega_{2q}, \omega_{2q}^{-1}) \,\, \mbox{ and } \,\,  \psi(T)= \,\, \left(\begin{smallmatrix}0&  1 \\ 1 &0\\ \end{smallmatrix}\right)$$ has Schur index $1$ and field of characters of degree $\tfrac{q-1}{2}$ over the rationals. It follows that the group algebra decomposition of $JS$ for each $S \in \mathscr{C}_g$ with respect to $\mathbf{D}_{2q}$ has the form \begin{equation}\label{hola6}JS \sim B^2 \times P\end{equation}where $B$ is the abelian subvariety of $JS$ associated to $\psi.$ The dimension of the fixed subspaces of $\psi$ under the action of $\langle R^q\rangle, \langle T \rangle, \langle TR\rangle$ and $\langle R^{q-1}\rangle$ equal  $0,1,1$ and $0$ respectively. Thereby, the equation  \eqref{dimensiones1} together with the fact that the action is represented by the ske $\theta=(R^q, T, TR, R^{q-1})$ imply that $$\dim B = \tfrac{q-1}{2}(-2+\tfrac{1}{2}\left((2-0)+(2-1)+(2-1)+(2-0))\right)=\tfrac{q-1}{2}$$and therefore $P=0.$ 
Now, we consider the subgroup $\langle T \rangle \cong C_2$ of $\mathbf{D}_{2q}$ and write $X=S / \langle T \rangle.$ The induced isogeny \eqref{decoind1} applied to \eqref{hola6} implies that $$JX=J(S/\langle T \rangle) \sim B \,\, \mbox{ and therefore }\,\, JS \sim JX^2.$$

A routine application of \eqref{eq:chevw} permits us to see that  the analytic representation $\rho_a$ of the action of $\mathbf{D}_{2q}$ on $S$ is equivalent to the Galois orbit of $\psi$;  namely$$\rho_a \cong \oplus_{\sigma} \psi^{\sigma}$$where $\sigma$ runs over the Galois group associated to character field of $\psi.$

\s

The following table (taken from \cite[Proposition 6.1]{IJR}) collects the character of ${\rho_a}$ and of its symmetric square for representatives of the conjugacy classes of the group.

\s
\s
\begin{center}
\begin{tabular}{|c|c|c|}
\hline
$g$ & $\chi_{\rho_a}$ & $\chi_{\rho_a}^{\sy}$ \\ \hline 
 1 & $q-1$ & $(q-1)^2+(q-1)$ \\ \hline
  $R^{2j-1}, \, j=1,\dots, q, j\neq \tfrac{q+1}{2}$ & 1& $1+(-1)$ \\ 
\hline
  $R^{2j}, \, j=1,\dots q-1$  & $-1$& $1+(-1)$ \\
    \hline
  $R^q$ & $-(q-1)$ & $(q-1)^2+(q-1)$ \\ \hline
   $T$ & 0 &  $0+(q-1)$ \\ \hline
  $TR$ & 0 &  $0+(q-1)$   \\ \hline
\end{tabular}
\end{center}
\s
\s
Thus, by the equation \eqref{Nsinrac}, we obtain that $$N_{S}=\tfrac{1}{8q}((q-1)^2+(q-1)+(q-1)^2+(q-1)+2q(q-1))=\tfrac{q-1}{2}.$$

\subsection*{The family ${\mathscr{K}}_g$} It is well-known  that $\mathbf{D}_q$ has two complex irreducible representations of degree one, and $\frac{q-1}{2}$ of degree two given by$$\psi_j: r\mapsto \text{diag}(\omega_q^j,\omega_q^{-j}) \text{ and } s\mapsto \left(\begin{smallmatrix} 0& 1 \\ 1&0\\ \end{smallmatrix}\right) \,\, \mbox{ where }\,\, j \in \{1, \ldots, \tfrac{q-1}{2}\};$$all of them are Galois conjugate. Clearly,   their direct sum $$W \cong \oplus_{j=1}^{\frac{q-1}{2}}\psi_j$$is a rational irreducible representation of $\mathbf{D}_q$ of degree $q-1.$  It follows that the group algebra decomposition of $JS$ for each $S \in \mathscr{K}_g-\mathscr{C}_g$ with respect to $\mathbf{D}_{q}$ has the form \begin{equation}\label{hola8}JS \sim B^2 \times P\end{equation}where $B$ is the abelian subvariety of $JS$ associated to $W.$ Observe that 
\begin{enumerate}
\item  The dimension of the fixed subspace of $\psi_1$ under  $\langle r^i\rangle$ equals 0, for each $1 \leqslant i \leqslant q-1$.
\item The dimension of the fixed subspace of $\psi_1$ under $\langle sr^i \rangle$ equals 1, for each $0 \leqslant i \leqslant q-1$.
\end{enumerate}

It follows that, independently of the equisymmetric stratum to which $S$ belongs (see the ske $\theta_i$ given in  the proof of Proposition \ref{ttex}), the equation  \eqref{dimensiones1} implies  $$\dim B= \tfrac{q-1}{2}[(-2+\tfrac{1}{2}((2-1)+(2-1)+(2-0)+(2-0))]=\tfrac{q-1}{2}.$$and therefore $P=0.$ 
Now, if $X=S / \langle s \rangle$ then  \eqref{decoind1} applied to \eqref{hola8} implies that $JS \sim JX^2.$

\s

By Lemma \ref{chew} one sees that the analytic representation $\rho_a$ of the action of $\mathbf{D}_q$ on $S$ is equivalent to $W$. The character of $\rho_a$ and of $\rho_a^{\sy}$ is summarised in the following table.

\s
\s

\begin{center}
\begin{tabular}{|c|c|c|}
\hline
$g$ & $\chi_{\rho_a}$ & $\chi_{\rho_a}^{\sy}$ \\ \hline 
 1 & $q-1$ & $(q-1)^2+(q-1)$ \\ \hline
  $sr^j, \, j=0,\dots, q-1$, & 0& $q-1$ \\ 
\hline
  $r^{j}, \, j=1,\dots, q-1$,  & $-1$& $-1+(-1)^2$ \\ \hline
\end{tabular}
\end{center}

\s
\s
Finally, by the equation  \eqref{Nsinrac} we obtain that 
$$N_{S}=\tfrac{1}{4q}\left((q-1)^2+(q-1)+q(q-1)\right)=\tfrac{q-1}{2}.$$

\section*{Addendum}

We recall here the fact that the full automorphism group of the Accola-Maclachlan curve $X_8$ determines its Jacobian variety $JX_8$ (that is, $N_{X_8}=0$) and therefore it allows us to determine its period matrix. In this addendum we determine explicitly the period matrix $$(I_4 \, Z) \,\, \mbox{ where } \,\, Z \in \mathscr{H}_4$$ of $JX_8$ provided that the genus of genus of $X_8$ equals four (that is, for $q=5$). To accomplish this task we apply the results on adapted hyperbolic polygons and algorithms  programed  in \cite{poly} to realise the action of the full automorphisms group of $X_8$ in the symplectic group. Explicitly, the rational representation $\rho_r : \mbox{Aut}(X_8) \to \mbox{Sp}(8, \mathbb{Z})$ is given by $$\rho_r(x^{-1})= \left(\begin{smallmatrix}   0 & 1&  1&  1& -1&  0&  0&  0\\ 1  &0&  0&  0&  0 &-1 & 0&  0\\ 0&  0&  0&  0&  0 & 1 &-1 & 0\\
 0&  0&  0&  0&  0&  0 & 1 &-1\\ 1&  0&  0&  0 & 0 & 0&  0&  0\\ 0 & 1 & 1 & 1 & 0 & 0&  0&  0\\ 0 & 0 & 1 & 1 & 0 & 0 & 0 & 0\\
 0 & 0&  0&  1 & 0 & 0&  0&  0\\ \end{smallmatrix}\right)
 \hspace{0.8 cm} \rho_r(zx)=\left(\begin{smallmatrix} 0 & 0&  1&  0 & 0 & 0 & 0 & 0\\ 1&  0&  0 & 0 & 0 &-1 & 0 & 0\\ -1 & 0 & 0&  0 & 0&  0 & 0 & 0\\
 0 & 0&  0&  0 & 0 & 0 & 0&  1\\ 0 & 0 & 0&  0&  0& -1 & 1 & 0\\ 0 & 1 & 1 & 0 & 0 & 0 & 0 & 0\\ 0&  1&  1&  0& -1&  0&  0&  0\\
  0 & 0&  0& -1&  0 & 0&  0 & 0\\ \end{smallmatrix}\right)
$$ Now, if we write $$Z=\left(\begin{smallmatrix}
a&b&c&d\\b&e&f&g\\c&f&h&j\\d&g&j&k\\ \end{smallmatrix}\right)$$then, with the notations of \S\ref{ff2}, the fact that $$ R \cdot Z = Z \,\, \mbox{ for each } \,\, R \in \langle \rho_r(x^{-1}), \rho_r(zx) \rangle$$ implies that the coefficients of $Z$  satisfy the  relations$$\begin{array}{lcl}
(1) \,\, a= -100/11 k^3 - 140/11 k, & \,\,\,\,\,  & (5) \,\, f= - 10/11k^3 - 39/22k, \\
(2) \,\, b= - 5k^2 - 31/4,& \,\,\,\,\, & (6)\,\,g= - 50/11k^3 - 151/22k, \\
(3) \,\, c=d=  5k^2 + 29/4, & \,\,\,\,\, & (7) \,\, h=k,\\
(4) \,\,e = 20/11k^3 + 39/11k, & \,\,\,\,\, & (8) \,\,j=  60/11k^3 + 84/11k,
\end{array}$$where the parameter $k$ satisfies the following equation
\begin{equation}\label{cigarro}k^4 + 5/2k^2 + 121/80=0.\end{equation}

The solutions of \eqref{cigarro} are$$k_1= -\tfrac{1}{2} i \, \sqrt{\tfrac{2}{5} \, \sqrt{5} + 5}, \,\, k_2= \tfrac{1}{2} i \, \sqrt{\tfrac{2}{5} \, \sqrt{5} + 5}, \,\, k_3= -\tfrac{1}{2} i \, \sqrt{-\tfrac{2}{5} \, \sqrt{5} + 5}, \,\, k_4= \tfrac{1}{2} i \, \sqrt{-\tfrac{2}{5} \, \sqrt{5} + 5}$$

The values $k_1$ and $k_4$ must be disregarded; indeed, 
$$a(k_1):= -100/11 k_1^3 - 140/11 k_1 \,\, \mbox{ and }\,\, a(k_4):= -100/11 k_4^3 - 140/11 k_4$$do not have imaginary part positive and therefore the corresponding matrices do not belong to $\mathscr{H}_4.$ Now, the fact that $Z \in \mathscr{H}_4$ also implies that $$ \delta := \mbox{det} \, \mbox{Im}\left(\begin{smallmatrix} a&b\\b&e\\ \end{smallmatrix}\right)  \mbox{ must be positive.}$$ With the help of numerical approximations of \cite{sage} one sees that  $\delta$ is positive only for $k=k_2.$

\s

We replace the value of $k=k_2$ in equalities $(1), \ldots, (8)$ to finally obtain that the period matrix of $JX_8$  is $(I_4 \, Z)$ where $Z$ is  given below.

{ \tiny 
$$\left(\begin{smallmatrix}
\frac{25}{22} i \, {\left(\frac{2}{5} \, \sqrt{5} + 5\right)}^{\frac{3}{2}} - \frac{70}{11} i \, \sqrt{\frac{2}{5} \, \sqrt{5} + 5} & \frac{1}{2} \, \sqrt{5} - \frac{3}{2} & -\frac{1}{2} \, \sqrt{5} + 1 & -\frac{1}{2} \, \sqrt{5} + 1 \\
\frac{1}{2} \, \sqrt{5} - \frac{3}{2} & -\frac{5}{22} i \, {\left(\frac{2}{5} \, \sqrt{5} + 5\right)}^{\frac{3}{2}} + \frac{39}{22} i \, \sqrt{\frac{2}{5} \, \sqrt{5} + 5} & \frac{5}{44} i \, {\left(\frac{2}{5} \, \sqrt{5} + 5\right)}^{\frac{3}{2}} - \frac{39}{44} i \, \sqrt{\frac{2}{5} \, \sqrt{5} + 5} & \frac{25}{44} i \, {\left(\frac{2}{5} \, \sqrt{5} + 5\right)}^{\frac{3}{2}} - \frac{151}{44} i \, \sqrt{\frac{2}{5} \, \sqrt{5} + 5} \\
-\frac{1}{2} \, \sqrt{5} + 1 & \frac{5}{44} i \, {\left(\frac{2}{5} \, \sqrt{5} + 5\right)}^{\frac{3}{2}} - \frac{39}{44} i \, \sqrt{\frac{2}{5} \, \sqrt{5} + 5} & \frac{1}{2} i \, \sqrt{\frac{2}{5} \, \sqrt{5} + 5} & -\frac{15}{22} i \, {\left(\frac{2}{5} \, \sqrt{5} + 5\right)}^{\frac{3}{2}} + \frac{42}{11} i \, \sqrt{\frac{2}{5} \, \sqrt{5} + 5} \\
-\frac{1}{2} \, \sqrt{5} + 1 & \frac{25}{44} i \, {\left(\frac{2}{5} \, \sqrt{5} + 5\right)}^{\frac{3}{2}} - \frac{151}{44} i \, \sqrt{\frac{2}{5} \, \sqrt{5} + 5} & -\frac{15}{22} i \, {\left(\frac{2}{5} \, \sqrt{5} + 5\right)}^{\frac{3}{2}} + \frac{42}{11} i \, \sqrt{\frac{2}{5} \, \sqrt{5} + 5} & \frac{1}{2} i \, \sqrt{\frac{2}{5} \, \sqrt{5} + 5}
\end{smallmatrix}\right)$$
}

\begin{remark} The results of \cite{poly} extend those of \cite{periodos}; thus, if we repeat this procedure for
the Accola-Maclachlan  curve of genus two, then we recover the period matrix given in \cite{periodos}.
\end{remark}

\s

\subsection*{Acknowledgements} The authors are very grateful to the referee for his/her  valuable comments and suggestions.

\end{document}